\documentclass[11pt,
final
]{amsart}
\usepackage{
amssymb, graphicx, tikz-cd
}
 \usepackage{color}
\newcommand{\ntext}{}

\usepackage{geometry}
\newgeometry{asymmetric, centering}
\usepackage{ifdraft}
\ifoptionfinal{
\usepackage[disable]{todonotes}
}{
\usepackage[bordercolor=white, color=white]{todonotes}
}
\newcommand{\HOX}[1]{\todo[noline, size=\footnotesize]{#1}}

\makeatletter\providecommand\@dotsep{5}\def\listtodoname{List of Todos}\def\listoftodos{\hypersetup{linkcolor=black}\@starttoc{tdo}\listtodoname\hypersetup{linkcolor=blue}}\makeatother

\date{\today}

\DeclareMathOperator{\diag}{diag}
\numberwithin{equation}{section}%

\newtheorem{theorem}{Theorem}[section]
\newtheorem{proposition}{Proposition}[section]
\newtheorem{lemma}{Lemma}[section]

\newtheorem{corollary}{Corollary}[section]

\theoremstyle{definition}

\DeclareMathOperator{\spa}{span}
\def\p{\partial}
\newcommand{\pair}[1]{\left\langle #1 \right\rangle}

\DeclareMathOperator{\supp}{supp}

\DeclareMathOperator{\WF}{WF}

\newcommand{\eps}{\varepsilon}

\newcommand{\R}{{\bf R}}
\newcommand{\Id}{\mbox{Id}}
\renewcommand{\r}[1]{(\ref{#1})}
\newcommand{\PDO}{$\Psi$DO}
\newcommand{\be}[1]{\begin{equation}\label{#1}}
\newcommand{\ee}{\end{equation}}

\renewcommand{\d}{\mathrm{d}}

\newcommand{\bo}{\partial M}

\renewcommand{\a}{a}

\def\X{\mathcal X}
\def\Y{\mathcal Y}
\def\CC{\mathcal C}

\title[The Light Ray transform]{The Light Ray transform on  Lorentzian manifolds}

\author[Lassas]{Matti Lassas}
\address{Matti Lassas, Department of Mathematics and Statistics, University of Helsinki, Box 68, Helsinki, 00014, Finland}
\thanks{ML partly supported by  Academy of Finland, grants 273979, 284715, 312110, 314879 
and the AtMath project of UH}
\author[Oksanen]{Lauri Oksanen}
\address{Lauri Oksanen, Department of Mathematics, University College London, Gower Street, London, WC1E 6BT, UK}
\thanks{LO partly supported by EPSRC Grant EP/P01593X/1 and EP/R002207/1.}
\author[P.~Stefanov]{Plamen Stefanov}
\address{Plamen Stefanov, Department of Mathematics, Purdue University, West Lafayette, IN 47907}
\thanks{PS partly supported by  NSF  Grant DMS-1600327}
\author[G. Uhlmann]{Gunther Uhlmann}
\address{Gunther Uhlmann, Department of Mathematics, University of Washington, Seattle, WA 98195, and IAS, HKUST, Clear Water Bay, Hong Kong, China}
\thanks{GU was partly supported by NSF}

\begin{document}
\begin{abstract}
We study the weighted light ray transform $L$ of integrating functions on  a Lorentzian manifold  over lightlike geodesics. We analyze $L$ as a Fourier Integral Operator and show that if there are no conjugate points, one can recover the spacelike singularities of {\ntext a function $f$}  {\ntext from its the weighted light ray transform $Lf$} by a suitable filtered back-projection. 
\end{abstract} 
\maketitle

\section{Introduction}   
Let $g$ be a Lorentzian metric with signature $(-,+,\dots,+)$ on the manifold $M$ of dimension $1+n$, $n\ge2$. We study the weighted  Light Ray Transform 
\be{1.1}
L_\kappa f(\gamma) = \int_\R  \kappa(\gamma(s), \dot\gamma(s)) f(\gamma(s))\,\d s,
\ee
of functions (or distributions) over light-like geodesics $\gamma(s)$, known also as  null geodesics. 
There is no canonical unit speed  parameterization as in the Riemannian case as discussed below, and we have some freedom to chose parameterizations locally by smooth changes of the variables. We are interested in microlocal invertibility of $L_\kappa$,
{\ntext that is, the description of which part of the singularities of the function $f$ can be reconstructed in a stable say when $L_\kappa f$ is given. Observe that this} property does not depend on the parameterization.  Here $\kappa$ is a weight function, positively homogeneous in its second variable of degree zero, which makes it parameterization independent.   
When $\kappa=1$, we use the notation $L$. 
 This transform appears in the study of hyperbolic equations when we want to recover a potential term, {\ntext or other coefficients of the equation,} from boundary or scattering information, see, e.g., \cite{MR1004174,Ramm-Sj, Ramm_Rakesh_91,waters2014stable, watersR_2013, Salazar_13, Aicha_15, Kian2016a,LOSU-strings, St-Yang-DN,KLU,Lassas, Feizmohammadi2019} for time dependent coefficients or in Lorentzian setting,  and also     \cite{BellassouedDSF, Carlos_12} for time-independent ones.  {\ntext This problem arises in medical ultrasound tomography (see Section 
 \ref{sec_applications} on applications for the details).
  The tensorial version of
 inverse problem for  the weighted  Light Ray Transform} 
  arises in the recovery of first order perturbations \cite{St-Yang-DN} and in linearized  problem of recovery a Lorentzian metric from remote measurements \cite{LOSU-strings}. The latter is  motivated by the problem of recovering the topological defects in the early stages of the Universe from the red shift data of the cosmic background radiation collected by the Max Planck satellite. 
The light tray transform $L$  belongs to the class of the restricted X-ray transforms  since the complex of geodesics is restricted to the lower dimensional manifold $g(\dot\gamma, \dot\gamma)=0$.  
 
The goal of the paper is to study the microlocal invertibility of $L_\kappa$ under some geometric conditions. Injectivity of $L$ on functions in the Minkowski case was proved in \cite{MR1004174}. 
Support theorems for analytic metrics and weights were proven in \cite{S-support2014}, see also \cite{Siamak2016} for a support theorem of $L$ on one-forms in the Minkowski case. Those results in particular  imply injectivity  under some geometric conditions. Microlocal invertibility or the lack of it however is important in order to understand the stability of that inversion. 
It is fairly obvious that $L_\kappa f$ cannot ``see'' the wave front set $\WF(f)$,
{\ntext of the function $f$,} in the timelike cone because $L_\kappa$ is smoothing there.  This just follows from the inspection of the wave front of the Schwartz kernel of $L_\kappa $, see also  Theorem~\ref{thm_M} for the Minkowski case. Microlocal invertibility for Minkowski metrics was studied in \cite{LOSU-strings, wang2017parametrices}. 
We show that in the general Lorentzian setting, one can recover  $\WF(f)$ in the spacelike cone if there are no conjugate points. {\ntext In relativistic setting, this roughly speaking means
that given $L_\kappa  f$, one can determine  the discontinuities (or the other singularities) of $f$ that move slower than the speed of light.}
Some restrictions are needed even in the Riemannian case.   One possible approach 
is to analyze  the normal operator $L_\kappa 'L_\kappa $ as in   \cite{Greenleaf-Uhlmann, Greenleaf_Uhlmann90, Greenleaf_UhlmannCM}. That operator is a Fourier Integral Operator (FIO) associated with two intersecting Lagrangians, 
see \cite{Greenleaf-U_90} and the references there for that class and the $I^{p,l}$ calculus of such operators. 
The analysis of $L_\kappa 'L_\kappa $  in the Minkowski case for $n=2$ is presented in \cite{Greenleaf-Uhlmann, Greenleaf_Uhlmann90, Greenleaf_UhlmannCM} as an example illustrating a much more general theory. Applying the $I^{p,l}$ calculus to get more refined microlocal results however requires the cone condition which cannot be expected to hold on general Lorentzian manifolds due to the lack of symmetry, as pointed out in \cite{Greenleaf_UhlmannCM}. We analyze $L_\kappa$ as an FIO and show that given any conically compact set $K$ in the spacelike cone,   one can choose a suitable  {\ntext pseudodifferential operator (\PDO)} cutoff $Q$ so that $L'QL$ is a \PDO\ elliptic in a neighborhood of $K$; therefore we can recover the singularities of $f$ from $X_\kappa f$ in $K$. 

The paper is organized as follows. In section~\ref{sec_M}, we analyze the flat Minkowski case where the formulas are more explicit. The Lorentzian case is studied in section~\ref{sec_Lorentz}, which contains our main results. In section~\ref{sec_2D}, we show that when $n=2$, singularities can actually cancel each other over pairs of conjugate points, similarly to the Riemannian case  \cite{MonardSU14}. 
In section~\ref{sec_applications}, we present two applications where the light ray transform appears naturally and our results can be applied: recovery of a time dependent potential in a wave equation in Lorentzian geometry and recovery of a linearization of a time dependent sound speed near a background stationary one.

\HOX{Matti: Add de Sitter. Lauri: We could recall the conformal invariance, and that Einstein-de Sitter reduces to Minkowski}
\section{The Minkowski case} \label{sec_M}
Let $g = -\d t^2+ (\d x^1)^2+\dots +(\d x^n)^2$ be the Minkowski metric in $\R^{1+n}$. 
Future pointing lightlike geodesics (lines) are given by 
\be{g}
\ell_{z,\theta}(s) = (s,z+s\theta)
\ee
with $z \in \R^n$ and $|\theta|=1$. 
This definition is based on parameterization of the lightlike geodesics by their point of intersection with the spacelike hypersurface $t=0$ and direction $(1,\theta)$. The parameterization $(z,\theta)$ defines a natural topology and a manifold structure of the set of the future pointing lightlike geodesics, which we denote by $\mathcal{M}$ below. 
We define the light ray transform
\be{Ldef}
Lf(z,\theta) = \int_\R f(s, z+s\theta)\,\d s,\quad z\in \R^{n}, \,\theta\in S^{n-1}.
\ee

The lightlike geodesics can be reparameterized by shifting and rescaling $s$. Our choice is based on having a unit orthogonal projection $\theta$ on $t=0$ but if we choose another spacelike hyperplane of hypersurface, this changes. Therefore, there is no canonical choice of the parameter along the lightlike lines. Note also that the notion of unit projection  $\theta$ is not invariantly defined under Lorentzian transformations, but in a fixed coordinate system, the scaling parameter $1$ (i.e., $\d t/d s=1$) is a convenient choice. 
More generally, we could use a parameterization locally near a lightlike geodesic $\gamma_0$, by choosing initial points on any hypersurface $S$ transversal to $\gamma_0$, and initial lightlike directions; and we can identify the latter with their projections onto  $S$. We will use such a choice in Section \ref{sec_Lorentz} below when we consider more general Lorentzian manifolds. 

Given  a weight $\kappa\in C^\infty(\R\times\R^n\times S^{n-1})$, we can define the weighted version $L_\kappa$ of $L$ by
\[
L_\kappa f(z,\theta) = \int_\R   \kappa(s, z+s\theta,\theta)  f(s, z+s\theta)\,\d s,\quad z\in \R^{n}, \,\theta\in S^{n-1}.
\]
Under a smooth change of the parameterization $s\mapsto\alpha(z,\theta)s$ with some $\alpha>0$,  the weight is transformed into a new one: $\kappa/\alpha$, and the microlocal properties we study remain unchanged. 

In the terminology of relativity theory, vectors $v=(v^0,v')$ satisfying $|v_0|<|v'|$ (i.e., $g(v,v)>0$) are called \textit{spacelike}. The simplest example are vectors $(0,v')$, $v'\not=0$. 
Vectors with $|v_0|>|v'|$ (i.e., $g(v,v)<0$) are \textit{timelike}; an example is $(1,0)$ which points along the time axis. 
\textit{Lightlike} vectors are those for which we have equality: $g(v,v)=0$. For covectors, 
the definition is the same but we replace $g$ by $g^{-1}$, which is consistent with the operation of raising and lowering the indices. Of course, in the Minkowski case  $g$ and $g^{-1}$ coincide. We say that a hypersurface is timelike, respectively spacelike, if its normal (which is a covector) is spacelike, respectively timelike.

We introduce the following three microlocal regions of $T^*\R^{1+n}\setminus 0$:
\begin{itemize}
\item[]spacelike cone, $\Sigma_{s}=\{(t,x;\tau,\xi);\; |\tau|<|\xi| \}$;
\item[]lightlike cone, $\Sigma_{l}=\{(t,x;\tau,\xi);\; |\tau|=|\xi| \}$;
\item[]timelike cone, $\Sigma_{t}=\{(t,x;\tau,\xi);\; |\tau|>|\xi| \}$.
\end{itemize}
In the Minkowski case, we can think of them as products of $\R^{1+n}$ and the corresponding cones in the dual space $\R^{1+n}$.

\begin{figure}[!ht] 
  \centering
  \includegraphics[width = 2in,page=2]{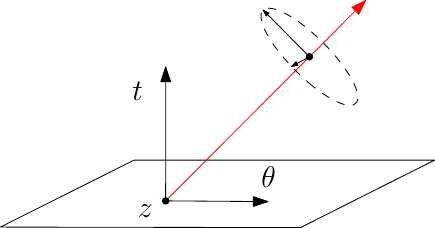}
  \caption{Knowing $Lf(z,\theta)$ for all $z$ and a fixed $\theta$ recovers the Fourier Transform of $f$ in codirections conormal to the lightlike lines in that set. Knowing it near some $(z,\theta)$ recovers $\WF(f)$  near those codirections along the line.}
  \label{pic01}
\end{figure}

\subsection{Fourier Transform analysis} 

By the Fourier Slice Theorem, knowing the X-ray transform for some direction $\omega$  recovers uniquely {\ntext the Fourier transform $\hat f$, of function $f$,} on $\omega^\perp$ if, say, $f$ is compactly supported. 
More precisely, the Fourier Slice Theorem in our case can be written as
\be{M1}
\hat f|_{ \tau+\xi\cdot\theta=0 } = 
\hat f(-\theta\cdot\xi,\xi) = \int_{\R^n} e^{-i z\cdot\xi}Lf(z,\theta) \, \d z, \quad\forall\theta\in S^{n-1}. 
\ee
The proof is immediate, and is in fact a consequence of the Fourier Slice Theorem in $\R^{1+n}$ for lines restricted to lightlike ones. The union of all $( 1,\theta)^\perp$ for all unit $\theta$ is $\{|\tau|\le |\xi|\} $ (the union $\Sigma_s \cup \Sigma_l$ of  the spacelike and the lightlike cones), as is easy to see. This correlates well with the theorems below. In particular, we see that knowing $L f$ for a distribution $f$ for which $Lf$ is well defined, and so is its Fourier transform, recovers $\hat f$ in  the spacelike cone $\Sigma_s$ uniquely and in a stable way. Under the assumption that $\supp f$ is contained in the cylinder $|x|\le R$ for some $R$ (and temperate w.r.t. $t$), one can use the analyticity of the partial Fourier transform of $f$ w.r.t.\ $x$ to extend $\hat f$ analytically to the timelike cone, as well. 
This is how it has been shown in  \cite{MR1004174} that $L$ is injective on such $f$. More general support theorems and  injectivity results, including such for analytic Lorentzian metrics, can be found in \cite{S-support2014}. 

\subsection{The normal operator $L'L$}
We formulate here a theorem about the Schwartz kernel of the normal operator $N=L'L$, where $L'$ is the transpose in terms of distributions (the same as the $L^2$ adjoint $L^*$ because the kernel of $L$ is real). The measure on $\R^n\times S^{n-1}$ is the standard product one. 
 One way  to prove the theorem is to think of $L$ as a weighted version of the X-ray transform $L$ with a distributional weight $2\sqrt{2}|\xi|\delta(\tau^2-|\xi|^2)$ and use the results about the weighted X-ray transform, see e.g. \cite{SU-Duke}, and allow a singular weight there. See also  \cite{SU-book,LOSU-strings}.

\begin{theorem}\label{thm_M}\ 
For every $f\in C_0^\infty(\R^{1+n})$, 

(a) 
\[
L'Lf = \mathcal{N}*f, \qquad \mathcal{N}(t,x) = \frac{\delta(t-|x|) +\delta(t+|x|) }{|x|^{n-1}}.
\]

(b) 
\[
L'Lf =  C_n\mathcal{F}^{-1}\frac{(|\xi|^2-\tau^2)_+^\frac{n-3}2} {|\xi|^{n-2}} \mathcal{F}f, \quad C_n:= 2\pi|S^{n-2}| . 
\]

(c) 
\[
h(\Box) f = C_n^{-1} |D_x|^{n-2}\Box_+^{\frac{3-n}{2}}L'Lf,
\]
where $h$ is the Heaviside function, and $\Box = \partial_t^2-\Delta_z$ and $\mathcal{F}$ is the Fourier transform. 
\end{theorem}

{\ntext Before proving Theorem \ref{thm_M}, we make some comments.}
Above, we used the notation $s_+^m=\max(s^m,0)$ with the convention  that $s_+^0$ is the Heaviside function. 
In particular, when $n=3$, we get $\sigma(L'L) =   4\pi^2|\xi|^{-1}h\left( |\xi|^2- \tau^2\right)$. Then
\[
h(\Box)f = (4\pi^2)^{-1} |D_z|L'Lf.
\]
As we can expect, there is a conormal singularity of the symbol even away from $\xi=0$ living on the characteristic cone. Moreover, $L'L$ is elliptic in the spacelike cone, and only there. This shows that $L'L$ is a formal \PDO\ with a singular symbol having singularities conormal to the light cone $\tau^2=|\xi|^2$, i.e., it is an FIO corresponding to two intersecting Lagrangians. This is one of the main examples in \cite{Greenleaf_UhlmannCM}. 
The theorem shows that ``singularities traveling slower than light'' can be recovered stably from $Lf$ known globally.  The ones traveling faster cannot.

\begin{proof}[Proof of Theorem \ref{thm_M}]
To compute the dual $L'$ of $L$, write
\[
\begin{split}
\langle Lf,\phi\rangle &=  \int_{S^{n-1}} \int_{\R^n}\int_\R f(s,x+s\theta)\phi(x,\theta) \,\d s\,\d x\,\d\theta\\
&=  \int_{S^{n-1}} \int_{\R^n}\int_\R f(s,x)\phi(x-s\theta,\theta) \,\d s\,\d x\,\d\theta.
\end{split}
\]
Therefore,
\be{X_Light_L'}
L'\phi(t,x) = \int_{S^{n-1}}  \phi(x-t\theta,\theta)  \,\d\theta, \quad  \phi\in C_0^\infty(\R^{n}\times S^{n-1}). 
\ee
In particular, this identity allows us to define $L$ on $\mathcal{E}'(\R^{1+n})$ by duality. 

By \r{1.1} and \r{X_Light_L'},
\[
\begin{split}
L'Lf(t,x) &= \int_{S^{n-1}} \int_\R f(s,x-t\theta+s\theta)\, \d s \, \d\theta\\
  & =  \int_{S^{n-1}}\left( \int_{s<t} +\int_{s>t} \right) f(s,x-t\theta+s\theta)\, \d s \, \d\theta.
\end{split}
\]
For the first integral, we get
\be{X_light_10.12}
\begin{split}
\int_{S^{n-1}} \int_{s<t}  f(s,x-t\theta+s\theta)\, \d s \, \d\theta 
  &= \int_{S^{n-1}} \int_{-\infty}^0  f(t+\sigma,x+\sigma\theta)\, \d \sigma \, \d\theta \\
 &= \int_{S^{n-1}} \int_0^{\infty}  f(t-\sigma,x+\sigma\theta)\, \d \sigma \, \d\theta\\
 &= \int_{\R^n} f(t-|z|,x+z)|z|^{1-n}\,\d z\\
 & = \int_{\R^n} \frac{f(t-|x-x'|, x')}{  |x-x'|^{n-1}}\,\d x'.
\end{split}
\ee
For the second one, we have
\[
\begin{split}
\int_{S^{n-1}} \int_{s>t}  f(s,x-t\theta+s\theta)\, \d s \, \d\theta 
 &= \int_{S^{n-1}} \int_0^{\infty}  f(t+\sigma,x+\sigma\theta)\, \d \sigma \, \d\theta\\
 & = \int_{\R^n} \frac{f(t+|x-x'|, x')}{  |x-x'|^{n-1}}\,\d x'.
\end{split}
\]
This completes the proof of (a). 
To prove (b), one can take formally the Fourier transform of $\mathcal{N}$ to get
\be{M2}
\hat{\mathcal{N}}(\tau,\xi) = 2\pi\int_{S^{n-1}} \delta(\tau+\theta\cdot\xi)\,\d\theta. 
\ee
This representation can also be justified by writing  \r{M1} in the form
\[
\hat f(-\theta\cdot\xi,\xi) = \mathcal{F}_{z\to\xi}Lf(z,\theta) \quad \Longrightarrow \quad Lf(z,\theta) =\mathcal{F}^{-1}_{\xi\to z} \hat f(-\theta\cdot\xi,\xi).
 \] 
Then 
\be{LL}
\begin{split}
\langle Lf,Lg\rangle &= (2\pi)^{-n} \int\int_{S^{n-1}} \hat f(-\theta\cdot\xi,\xi) \hat g(-\theta\cdot\xi,\xi)\,\d\theta \,\d\xi\\
& = (2\pi)^{-n}\int\int_{S^{n-1}} \delta(\tau+ \theta\cdot\xi)  \hat f(\tau,\xi) \hat g(\tau,\xi)\, \d\theta \,\d\tau\, \d\xi .
\end{split}
 \ee
Therefore, if we denote for a moment by $K$ the integral in \r{M2} but multiplied by $(2\pi)^{-n}$ instead of $2\pi$, we get $\langle L\mathcal{F}^{-1} \hat f, L\mathcal{F}^{-1}\hat g \rangle =  \langle \hat f,  K \hat g\rangle $; hence $\mathcal{F^*}^{-1}L'L\mathcal{F}^{-1}=K  $. Since $\mathcal{F}^* = (2\pi)^{1+n}\mathcal{F}^{-1}$, we get \r{M2}. 

To compute $\hat{\mathcal{N}}$ explicitly, take a test function $\phi(\tau,\xi)$ and write
\[
\begin{split}
\langle \hat{\mathcal{N}} ,\phi\rangle   &= 2\pi\int \int_{S^{n-1}} \delta(\tau+\theta\cdot\xi) \phi(\tau,\xi) \,\d\theta\, \d\tau\,\d\xi = 2\pi\int \int_{S^{n-1}} \phi(-\theta\cdot\xi,\xi) \,\d\theta\,\d\xi \\
  &= 2\pi \iint F(s,\xi)\phi(s,\xi) \, \d s\, \d\xi
\end{split}
\]
with $F$ is the $L^1_{\rm loc}$ function  in \r{MF} below. This proves (b). 

Part (c) of the lemma follows directly from (b). 
\end{proof}

We used the following lemma. 
\begin{lemma} \label{X_Light_lemma_int}
For every $\psi\in\mathcal{S}(\R^{1+n})$,
\be{X_light_l1}
\int_{S^{n-1}} \psi(\theta\cdot\xi)\,\d\theta =  |S^{n-2}||\xi|^{2-n}\int_\R\psi (s) (|\xi|^2-s^2)_+^{\frac{n-3}{2}} \,\d s,\quad \xi\not=0,
\ee
where $|S^{n-2}|$ is the area of $S^{n-2}$ if $n\ge3$; equal to $2$ when $n=2$. 
\end{lemma}

Note that the kernel 
\be{MF}
F(s,\xi):= |S^{n-2}||\xi|^{2-n} (|\xi|^2-s^2)_+^{\frac{n-3}{2}}
\ee
is homogeneous of order $-1$ and as such, it is locally integrable. It has a unique extension as a homogeneous distribution of order $-1$ given by an $L^1_{\rm loc}$ function. Also, the l.h.s.\ of \r{X_light_l1} is a smooth function of $\xi$ everywhere, including at $\xi=0$. 

\subsection{$L_\kappa$ as an FIO} \label{sec_L_FIO}
Theorem~\ref{thm_M}(c) implies some recovery of singularities results already. If $f\in \mathcal{E}'(\R^{1+n})$ and $L_\kappa f\in C^\infty(\R^n \times S^{n-1})$, then $\WF(f)$ does not contain spacelike singularities (note that this argument requires global knowledge of $L_\kappa f$). One the other hand, one can easily construct functions of distributions with timelike singularities so that $L_\kappa f=0$; for example take any non-smooth $h\in \mathcal{E}'(\R)$ with integral zero, then for any $a\in\R^n$ with $|a|<1$,  for $f= h(t+x\cdot  a)$ we have $L f=0$; and 
\[
\WF(f)=
{\ntext  \{(t ,x,\tau,\xi )\ |\; t=s-x\cdot a,\ \xi=a\tau,\ (s,\tau)\in \WF(h) \}.}
\]
Then $|\xi|=|\tau| |a|<|\tau|$ is in the timelike cone. In particular, $\delta'(t)$ is in the kernel of $L_\kappa$ and has timelike singularities only. 

 We can get more precise statements by studying first the Schwartz kernel of $L_\kappa $. It is given by 
\be{L_kernel}
\mathcal{L_\kappa}(z,\theta,t,x) = \kappa(t, x,\theta) \delta( x- z-t\theta ).
\ee
In other words, $\mathcal{L_\kappa} = \kappa\delta_X$, where 
\[
X = \left\{ (z,\theta,t,x)|\; 
x=z+t\theta\right\} 
\]
is the point-line relation. 
We write $M=\R^{1+n}= \R_{t,x}^{1+n}$ and let $\mathcal M \cong \R^n_z\times S_\theta^{n-1}$ be the manifold of the lines in $M$.
Clearly, $X$ is a $2n$-dimensional submanifold of the product $\mathcal{M}\times    M \cong \R^n_z\times S_\theta^{n-1}\times \R_{t,x}^{1+n}$ which itself is $3n$-dimensional. Its conormal bundle is given by 
\[
N^*X = \left\{  ((z,\theta,t,x) , (\zeta,\hat\theta, \tau,\xi )     )\big| \; x=z+t\theta,\,  \xi=-\zeta, \,  \tau = -\theta\cdot\xi, \, \hat\theta= t(-\xi +(\xi\cdot\theta) \theta)  \right\}
\]
with $\hat\theta$ conormal to $S^{n-1}$ at  $\theta$. 
We consider $N^*X$ as a subset of $T^*(\mathcal{M}\times    M) \setminus 0$.
This is a conical Lagrangian manifold which coincides with the wave front set of the kernel $\mathcal{L}_\kappa$ when $\kappa$ is nowhere vanishing; and includes the latter for general $\kappa$. 

Note that $(\tau,\xi)$ is space or light-like on $N^*X$ and it is the latter if and only if 
    \begin{align}\label{minkowski_lightlike}
\zeta \parallel \theta.
    \end{align}
Indeed, $|\tau| = |\xi|$ is equivalent with $|\theta \cdot \zeta| = |\zeta|$ on $N^* X$.
As will be explained below, the relation (\ref{minkowski_lightlike}) allows us to choose a microlocal cutoff on $\mathcal M$ so that when applied to $L_\kappa f$, it cuts away the singularities in $\WF(f)$ near $\Sigma_l$. This will be useful in view of the singular behavior near $\Sigma_l$, as illustrated for $L' L$ in Theorem \ref{thm_M}.

Let us also mention that $|\tau| = |\xi|$ is equivalent with $-\xi +(\xi\cdot\theta) \theta = 0$ on $N^* X$.
In particular, $\hat \theta = 0$ in this case. 
We will show, see Lemma \ref{lem_lightlike} below, 
that on general Lorentzian manifolds, $(\tau,\xi)$ being lightlike on $N^* X$ is equivalent with $\zeta \parallel \theta$ and $\hat \theta = 0$ (or rather its suitable reformulation in the more general context). 

The canonical relation associated to $L_\kappa$ is given by 
\be{C0}
\begin{split}
C := N^*X' = \Big\{  &((z,\theta,\zeta,\hat\theta) , (t, x, \tau,\xi )     )
{\ntext \in T^*(\mathcal{M}\times M) \setminus 0}\ 
\big| \\ &\; x=z+t\theta, \,  \tau = -\theta\cdot\xi, \, \zeta=\xi, \, \hat\theta=t(\xi -(\xi\cdot\theta) \theta ) \Big\}.
\end{split}
\ee
Here we rearranged the variables to comply with the notational convention in \cite{Hormander4}. 
If one of the covectors $(\zeta,\hat\theta)$ and $(\tau,\xi)$ vanishes, then the other one does, too. Therefore, $C$ is a (homogeneous) canonical relation from $T^*M\setminus0$ to $T^*\mathcal{M}\setminus0$ 
 and it is also clearly conically closed in $T^*(\mathcal{M}\times M) \setminus 0$. 
Therefore, this, and the fact that its kernel is a conormal distribution, show that  $L_\kappa$ is an FIO with the canonical relation $C$, see \cite[Chapter~XXV.25.2]{Hormander4}. 
In particular, 
\be{WF}
\WF(L_\kappa f) \subset C\circ \WF(f),
\ee
a statement independent of the FIO theory. In order to compute the order of $L_\kappa$, we can write its Schwartz kernel as the oscillatory integral
$$
\frac {\kappa(t,x,\theta   ) } {(2\pi)^n} \int_{\R^n} e^{i(t\theta\cdot\xi+(z-x)\cdot\xi)} \d\xi,
$$
see \r{L_kernel}. Then the order $m$ of $L_\kappa$ satisfies, see \cite[Def. 3.2.2]{Hormander_FIO},
    \begin{align}\label{order_min}
0 = m + (\dim(M) + \dim(\mathcal M) - 2n)/4,
\quad \text{that is,} \quad
m = -n/4,
    \end{align}
because $M = \R^{1+n}$ and $\mathcal M \cong \R^n \times S^{n-1}$.

The relation $C$ also allows the following interpretation: it consists of points and lightlike lines through them; next, $(\tau,\xi)$ is conormal to $(1,\theta)$, i.e., to each such line $\ell_{z,\theta}$; and the dual variables $(\zeta,\hat\theta)$ can be interpreted as  projections of Jacobi fields along  the line $\ell_{z,\theta}$ to its conormal bundle.
This interpretation is discussed further in Section \ref{sec_Lorentz} below in the context of general Lorentzian manifolds, see (\ref{Pi_M1}).

Let $\pi_\mathcal{M}$, $\pi_M$ be the natural projections of $C$ onto $T^* \mathcal{M}$ and $T^*M$, respectively. 
\be{4.2}
\begin{tikzcd}[]
&  C \arrow{dl}[swap]{\pi_{\mathcal M}}\arrow{dr}{\pi_M} & \\
 T^*\mathcal{M}&&T^*M 
\end{tikzcd}
\ee
The dimensions from left to right are $4n-2 \ge 3n\ge 2n+2$. The difference between two consecutive terms is $n-2$ and they are all equal when $n=2$. 
The manifold $Z$ can be parameterized by $(z,\theta,t)$. Then $C$ can be parameterized by 
$$
C_0 = \{(z,\theta,t,\xi) \in \R^n \times S^{n-1} \times \R \times (\R^n \setminus 0)\}.
$$. 

We have  
\be{pi_calM}
\pi_\mathcal{M} ((z,\theta,\zeta,\hat\theta) , (t, x, \tau,\xi )) =     (z,\theta,\zeta,\hat\theta)   = (z,\theta,\xi, t(\xi -(\xi\cdot\theta) \theta ) ) .
\ee
This is a map from the $3n$ dimensional $C$ to the $4n-2$ dimensional $T^*\mathcal{M}$. 
If $n=2$, $\pi_\mathcal{M}$ is a local diffeomorphism when $C$ is restricted to spacelike $(\tau, \xi)$.  Indeed, we recall that in that case $\xi -(\xi\cdot\theta)\theta\not=0$. Therefore, 
the equation $\hat \theta =  t(\xi -(\xi\cdot\theta)\theta)  $ can be solved for $t$. 
When $n\ge3$, $\d\pi_\mathcal{M}$ has full rank $3n$ away from the light-like cone, i.e., the defect is $(4n-2)-3n= n-2$ and in particular is injective  there.  The projection $\pi_\mathcal{M}$ is also injective, therefore, it is an immersion (on the spacelike cone). 
 Next, there is $t \in \R$ such that $\hat \theta =  t(\xi -(\xi\cdot\theta)\theta)$ if and only if $\hat\theta$ is colinear with the projection of  $\zeta= \xi$ to $\theta_\perp= \{\xi|\; \xi\cdot\theta=0\}$, which describes the range of $\pi_\mathcal{M}$ for $(\tau,\xi)$ spacelike.

If $(\tau,\xi)$ is lightlike, then the right-hand side of (\ref{pi_calM}) reduces to $(z,\theta,\xi,0)$. In particular, for lightlike $(\tau^j,\xi^j)$, $j=1,2$, the equation
    \begin{align}\label{pi_calM_lightlike}
\pi_\mathcal{M} ((z,\theta,\zeta,\hat\theta) , (t_1, x_1, \tau^1,\xi^1 )) =  \pi_\mathcal{M} ((z,\theta,\zeta,\hat\theta) , (t_2, x_2, \tau^2,\xi^2 ))
    \end{align}
is equivalent with $( \tau^1,\xi^1) = ( \tau^2,\xi^2)$ and both $(t_j, x_j)$, $j=1,2$, lying on the line $\ell_{z,\theta}$.

For the second projection $\pi_M$ in \r{4.2} we get
    \begin{align}\label{pi_M}
\pi_{M} ((z,\theta,\zeta,\hat\theta) , (t, x, \tau,\xi )) =     (t, x, \tau,\xi )  = (t, z+t\theta, -\theta\cdot\xi, \xi) .
    \end{align}
Its differential has full rank for spacelike $(\tau,\xi)$.  The projection $\pi_M$ is surjective onto the spacelike cone, as well. Indeed, given $(t,x,\tau,\xi) $, we need to solve the equation given by the second equality above  for the parameters  $(z,t,\theta,\xi)$. 
The variables $t$ and $\xi$ are obtained trivially, and we need to solve $x=z+t\theta$ and $\tau=-\theta\cdot\xi$ for $z$ and $\theta$.  For unit $\theta$, the latter equation    has an $(n-2)$ dimensional sphere of solutions (the intersection of the unit sphere with that plane in the $\theta$ space) when  $(\tau,\xi)$ is spacelike. For each solution $\theta$, we obtain $z$ by solving $z+t\theta=x$.    We can choose a locally smooth solution, which in particular shows that the differential has full rank.
If $(\tau,\xi)$ is lightlike, i.e., if $|\tau|=|\xi|$, the equation   $\tau= -\theta\cdot\xi$ has a unique solution for $\theta$ given by $\theta=-\text{sgn}(\tau)\, \xi/|\xi|$. If $(\tau,\xi)$ is timelike, there are no solutions. 

If $n=2$, $\pi_M$ is a local diffeomorphism and it is $2$-to-$1$ in the spacelike cone  because   $\tau= -\theta\cdot\xi$  has two solutions for $\theta\in S^1$:  $\theta_\pm  = \pm \sqrt{1-\tau^2/|\xi|^2 }( \xi^\perp/|\xi| ) -\tau\xi/|\xi|^2$ for spacelike $(\tau,\xi)$ with some fixed choice of the rotation by $\pi/2$ to define $\xi^\perp$.  This describes the non-uniqueness class of $\pi_M$.

We summarize the properties of the projections $\pi_{\mathcal M}$ and $\pi_M$ as follows. 

\begin{lemma}\label{lem_pr_min}
The differential $\d\pi_\mathcal{M}$ is injective and the differential $\d \pi_M$ is surjective at $(z,\theta,t,\xi) \in C_0$, with $\xi$ spacelike.
The projection $\pi_{\mathcal M}$ is injective on the set of points $(z,\theta,t,\xi) \in C_0$, with $\xi$ spacelike. 
The projection $\pi_M$ is surjective onto $\Sigma_s$.
\end{lemma}

Let us also summarize the properties of $C$ considered as a relation (a multi-valued map $C=\pi_\mathcal{M}\circ \pi_M^{-1}$).

\begin{lemma}\label{lem_C_min_1}  
$C$ has domain $\Sigma_s\cup \Sigma_l$. For every $(t,x,\tau, \xi)\in \Sigma_s$, $C(t,x,\tau, \xi)$ is the set of all $(x-t\theta,\theta,\xi, t(\xi -(\xi\cdot\theta) \theta ) )$ with $\theta\in S^{n-1}$ a   solution of $(\tau,\xi)\cdot (1,\theta)=0$. 
\begin{itemize}
\item[(a)] If $n=2$, then $C$ is a local diffeomorphism from $\Sigma_s$ to $T^*\mathcal{M}\setminus 0$, and a $1$-to-$2$ map globally on $\Sigma_s$. 
\item[(b)] If $n\ge3$, for every $(t,x,\tau, \xi)\in \Sigma_s$,  $C(t,x,\tau, \xi)$ is diffeomorphic to $S^{n-2}$.  
\end{itemize}
For every $(t,x,\tau,\xi)\in\Sigma_l$, $C(t,x,\tau, \xi)= (x-t\theta,\theta,\xi,0)$, where $\theta= -\text{\rm sgn}(\tau)\xi/|\xi|$. In particular, $C(\Sigma_l)= \big\{ (z,\theta,\zeta,\hat\theta)\,\big|\;\zeta\parallel \theta, \; \hat\theta=0\big\}\setminus 0$.
\end{lemma}

In particular, this proposition says that $\WF(f)$ in the spacelike cone may affect $\WF(L_\kappa f)$ at all lightlike lines $\ell$ through the base point and normal to its the covector there. 

The properties of $C^{-1}$ are summarized as follows.

\begin{lemma}\label{lem_C_min_2}  
$C^{-1}$ has domain in $T^*\mathcal{M}\setminus 0$ consisting of all $(z,\theta, \zeta,\hat\theta)$ so that 
$\hat\theta$ is colinear with the projection of  $\zeta$ to $\theta_\perp$. Its range is $\Sigma_s\cup \Sigma_l$. The points mapped to $\Sigma_l$ are the ones with $\theta\parallel\zeta$. 
For every $(z,\theta,\zeta, \hat\theta) $ in the domain with $\theta  \not\hspace{.14em}\parallel \zeta$, we have  
$$C^{-1}(z,\theta,\zeta, \hat\theta) = (t,z+t\theta,-\theta\cdot\zeta,  \zeta),$$ where $t$ is the unique solution to $\hat\theta = t(\xi-(\xi\cdot\theta)\theta)$. 
\end{lemma}

When $n=2$, the colinearity condition is automatically satisfied. Indeed, the space $\theta_\perp$ is one dimensional then and therefore any $\hat \theta \in T_\theta^* S^1 \cong \theta_\perp$ is colinear with the projection of  $\zeta$ to $\theta_\perp$. 
When $n\ge3$, unlike $C$, the relation $C^{-1}$ is a map away from $\theta\parallel\zeta$. It is not injective by Lemma~\ref{lem_C_min_1}. 

Most importantly for the purposes of the present paper, the composition $C^{-1} \circ C$ reduces to the identity on $\Sigma_s$. This can be deduced directly from Lemma~\ref{lem_pr_min}, as will be done in the proof of Lemma~\ref{lem_CC_comp} in the general Lorentzian context, however, we will give here a proof based on Lemmas \ref{lem_C_min_1} and \ref{lem_C_min_2}.

\begin{lemma}\label{lem_CC_comp_min}
For every $(t,x,\tau, \xi)\in \Sigma_s$ it holds that  
$(C^{-1} \circ C)(t,x,\tau, \xi) = (t,x,\tau,\xi)$.
\end{lemma}
\begin{proof}
From Lemma~\ref{lem_C_min_1},
\[
C(t,x,\tau, \xi) = \{ (z,\theta,\xi,\hat \theta)\,|\, z = x-t\theta,\ \hat \theta = t(\xi-(\xi\cdot\theta) \theta),\ \theta \in S^{n-1},\ -\theta \cdot \xi = -\tau\},
\]
and from Lemma~\ref{lem_C_min_2}, for $(z,\theta,\xi,\hat \theta) \in C(t,x,\tau, \xi)$,
\[
C^{-1}(z,\theta,\xi,\hat \theta) = (t,x,\tau,\xi).
\]
\end{proof}

\subsection{Recovery of spacelike singularities}

Lemma~\ref{lem_CC_comp_min} suggests that the composition of $L_\kappa$ with its transpose $L_\kappa'$ could be a pseudodifferential operator when restricted on $\Sigma_s$. On the other hand, by
Propositions~\ref{lem_C_min_1} and \ref{lem_C_min_2}, $C$ maps the lightlike cone to $\{\zeta\parallel\theta\}$, and $C^{-1}$ maps the latter to the former. As anticipated above, this suggests that  we could cut the data $L_\kappa f$ microlocally near $\{\zeta\parallel\theta\}$ to apply a cutoff to $\WF(f)$ near $\Sigma_l$. This is not an automatic application of Egorov's theorem however because $C$ and $C^{-1}$ are singular near the lightlike cone (and its image under $C$) and $L_\kappa$ is not a classical FIO there, in sense that the associated canonical relation is not a canonical graph. Next theorem gives local recovery of space like singularities from local data. It is similar to Proposition~{11.4} in our previous paper \cite{LOSU-strings}. 

\begin{theorem} \label{th_rec_min}
Let $Q=q(z,\theta,D_z)$ be a \PDO\ in $\mathcal M$ with a symbol $q(z,\theta,\zeta)$ of order zero (independent of $\hat\theta$) supported in $\{|\theta\cdot\zeta|<|\zeta|\}$. Then $L_\kappa 'QL_\kappa$ is a \PDO\ in $M$ of order $-1$   with essential support in the spacelike cone.  

Suppose, moreover, that $\kappa$ is nowhere vanishing.
Let $U \subset \R^n \times S^{n-1}$ be a neighborhood of $(z_0, \theta_0) \in \R^n \times S^{n-1}$, and let $(t_0, x_0, \tau^0, \xi^0) \in \Sigma_s \cap N^* \ell_{z_0, \theta_0}$. Then $Q$ can be chosen so that its essential support is contained also in $U \times (\R^n \setminus 0)$ and that $L_\kappa' Q L_\kappa$ is elliptic at $(t_0, x_0, \tau^0, \xi^0)$.
\end{theorem}
\begin{proof}
The Schwartz kernel $\mathcal{L}_\kappa$ of $L_\kappa$ has a wave front set $N^*Z$ and $C$ is its relation, see also \r{WF}. We can always assume that the essential support $\text{ess-sup}( Q)$ of $Q$ is conically compact. 
The twisted wave front set of the Schwartz kernel of $Q$ as a relation is identity restricted to $\text{ess-sup}( Q)$. Then its composition with the relation $C$ is $C$ again with its image restricted by $Q$ to $\text{ess-sup}( Q)$ which is contained in 
the conic set $\{|\theta\cdot\zeta|<|\zeta|\}$. By \r{C0}, this implies $|\tau|<|\xi|$ near the wave front of the kernel of $QL_\kappa$. Therefore, $QL_\kappa$ is smoothing in a conic neighborhood of  $\Sigma_l\cup\Sigma_t$, and so is $L_\kappa'QL_\kappa$. 

The composition $L_\kappa'QL_\kappa$ can be analyzed by using the
the transversal intersection calculus in the case $n=2$, and 
the clean intersection calculus in the case $n=3$.
As the composition $C^{-1} \circ C$ is the identity on $\Sigma_s$, the calculi imply that $L_\kappa'QL_\kappa$ is a \PDO\ of order $-1$.
We will focus on the more complicated case $n \ge 3$, and justify the application of the clean intersection calculus in the next section.

Writing $\sigma[\cdot]$ for the principal symbol map, it holds that $\sigma[L_\kappa'QL_\kappa]$ is obtained from $\sigma[L_\kappa'] = \kappa$, $\sigma[Q]$ and $\sigma[L_\kappa] = \kappa$ by an integration reducing the excess, see \cite[Theorem~25.2.3]{Hormander3}.
We will choose $Q$ so that $\sigma[Q]$ is non-negative.
As $\kappa$ is nowhere vanishing,  $\sigma[L_\kappa'QL_\kappa]$ is positive at $(t_0, x_0, \tau^0, \xi^0)$ if and only if the integral of $\sigma[Q]$ does not vanish over the fiber $C(t_0, x_0, \tau^0, \xi^0)$.

We set $\zeta^0 = \xi^0$ and choose $Q$ so that $\sigma[Q](z_0, \theta_0, \zeta_0) > 0$. It holds that 
$$
(z_0, \theta_0, \zeta^0, \hat \theta^0) \in C(t_0, x_0, \tau^0, \xi^0),
$$
where $\hat \theta^0 = t_0 (\xi_0 - (\xi_0 \cdot \theta_0)\theta_0)$.
Indeed, this follows from Lemma~\ref{lem_C_min_1} since the assumption $(t_0, x_0, \tau^0, \xi^0) \in N^* \ell_{z_0, \theta_0}$ implies that $x_0 = z_0 + t_0 \theta_0$
and $(\tau^0, \xi^0) \cdot (1,\theta_0) = 0$.
Therefore, $\sigma[Q]$ does not vanish identically on $C(t_0, x_0, \tau^0, \xi^0)$.

It still remains to show that the choice $\sigma[Q](z_0, \theta_0, \zeta_0) > 0$ is compatible with the requirement that $\text{ess-sup}(Q) \subset \{|\theta\cdot\zeta|<|\zeta|\}$. 
This follows, since together with $\zeta^0 = \xi^0$ and $(t_0, x_0, \tau^0, \xi^0) \in \Sigma_s$, the orthogonality $(\tau^0, \xi^0) \cdot (1,\theta_0) = 0$ implies that 
$$
|\theta_0 \cdot \zeta^0| = |\tau^0| < |\xi^0| = |\zeta^0|.
$$
\end{proof}

As a corollary, we have the following global result saying that the space like singularities can be recovered.

\begin{corollary}
Let $L_\kappa f\in C^\infty(\mathcal M)$ and assume that $\kappa$ vanishes nowhere. Then it holds that $\WF(f)\cap \Sigma_s=\emptyset$. 
\end{corollary}
\begin{proof}
For any $(t_0, x_0, \tau^0, \xi^0) \in \Sigma_s$ we can choose a lightlike line $\ell_{z_0, \theta_0}$ such that $(t_0, x_0, \tau^0, \xi^0)$ is in $N^* \ell_{z_0, \theta_0}$.
Then the previous corollary implies that $(t_0, x_0, \tau^0, \xi^0) \notin \WF(f)$.
\end{proof}

By combining Theorem~\ref{th_rec_min} with a microlocal partition of unity, we can recover, not only $\WF(f)$, but a smoothened version of $f$ with the singularities cut off (in a smooth way) in any predetermined neighborhood of $\Sigma_t\cup\Sigma_l$. This can be viewed as a regularized inversion of $L_\kappa$ with the regularization cutting away from  the ill posed region $\Sigma_t$ and its boundary $\Sigma_l$. 

Let us also give a more explicit construction as follows. We can choose  $\phi \in C_0^\infty(\R)$ such that $\phi=1$ on $[0,1-\eps]$ and $\phi=0$ on $[1-\eps/2,\infty)$. Let $Q$ be the zeroth order \PDO\ with  symbol $\phi(|\theta\cdot\zeta|/|\zeta|)$ cut off smoothly near the origin (which is actually not needed). Then $L_\kappa'QL_\kappa$ is a \PDO\ of order $-1$, elliptic away from a neighborhood of $\Sigma_t \cup \Sigma_l$ determined by $\epsilon$. 
When $\kappa=1$,   one can compute $L'QL$ directly. Since $|\theta\cdot\zeta|/|\zeta|$ is a Fourier multiplier w.r.t.\ $z$ only, it is enough to express $QLf$ by taking the Fourier transform of $Lf$ w.r.t.\ $z$ only. Then from \r{LL} we get
\[
L'QL = \mathcal{F}^{-1}\frac{\phi(|\tau|/|\xi|)}{|\xi|}   \left(|1-\tau^2/|\xi|^2\right)^\frac{n-3}2 \mathcal{F}. 
\]
Therefore, 
\[
\phi(|D_t|/|D_x|) f =   |D_x|\left(1-D_t^2/|D_x|^2\right)_+^\frac2{n-3} L'QLf.
\]

\subsection{The clean intersection calculus}\label{sec_clean_intersection}

We assume that $n \ge 3$, and show here that the clean intersection calculus can be applied to $L'_\kappa Q L_\kappa$ in Theorem~\ref{th_rec_min}. 
The traditional formulation of this calculus considers the composition $A_1 A_2$ of two properly supported Fourier integral operators $A_1$ and $A_2$ such that the composition $C_1 \circ C_2$ of their canonical relations 
$$
C_1 \subset (T^*X \setminus 0) \times (T^*Y \setminus 0),
\quad
C_2 \subset (T^*Y \setminus 0) \times (T^*Z \setminus 0),
$$
is clean, proper and connected
\cite[Th. 25.2.3]{Hormander4}.
Here $X$, $Y$ and $Z$ are smooth manifolds.
The operators $A_1 = L_\kappa'$ and $A_2 = Q L_\kappa$ do not quite 
satisfy the assumptions of the calculus, since the composition $C^{-1} \circ C$ is clean only away from $\Sigma_l$.
Also, as a canonical relation, $C$ must be closed in $T^* (\mathcal M \times M) \setminus 0$, and we can not simply apply the calculus with $C$ replaced by $C \setminus (T^* \mathcal M \times \Sigma_l)$. 

The proof of \cite[Th. 25.2.3]{Hormander4} uses a microlocal partition of unity, subordinate to a cover $\Gamma_j$, $j=1,2,\dots$, of the intersection $\X \cap \Y$ where
$$
\X = C_1 \times C_2, \quad \Y = T^* X \times \diag(T^* Y) \times T^* Z,
$$
and $\diag(T^* Y) = \{(p,p);\ p \in T^* Y\}$. We write $K_1$ for the Schwartz kernel of $A_1$, and recall that the essential support of $A_1$ is given by 
$$
\text{ess-sup}(A_1) = \WF'(K_1) = \{(x,\xi,y,-\eta)\,|\, (x,\xi,y,\eta) \in \WF(K_1)\}.
$$
For the local step of the proof, it is enough to assume that the composition $C_1 \circ C_2$ is clean in each $\Gamma_j$ that intersect the product $\text{ess-sup}(A_1) \times \text{ess-sup}(A_2)$. 
The composition $C_1 \circ C_2$ being clean in $\Gamma_j$ means that 
$\X \cap \Y \cap \Gamma_j$ is a smooth manifold and that 
\begin{align}
\label{clean_intersection_Tp}
T_p (\X \cap \Y) = T_p \X \cap T_p \Y, \quad p \in \X \cap \Y \cap \Gamma_j.
\end{align}

The local step implies that $(C_1 \circ C_2)'$ is locally a conic  Lagrangian manifold, however, global assumptions are needed, for example, to guarantee that it does not have self-intersections.  
The assumptions that $C_1 \circ C_2$ is proper and connected are used in the proof \cite[Th. 25.2.3]{Hormander4} to show that $C_1 \circ C_2$ is an embedded submanifold of $T^* (X \times Z) \setminus 0$, and closed as its subset. 

In our case,
$A_1 = L_\kappa'$ and $A_2 = Q L_\kappa$,
$$
\X = C^{-1} \times C, \quad \Y = T^* M \times \diag(T^* \mathcal M) \times T^* M,
$$
and $\text{ess-sup}(A_2) \subset \Sigma_s$ due to the microlocal cutoff $Q$. Thus we need to consider the condition (\ref{clean_intersection_Tp}) only for 
    \begin{align}\label{splike_cover}
\Gamma_j \subset T^* M \times \diag(T^* \mathcal M) \times \Sigma_s
    \end{align}
As for the global structure of $C^{-1} \circ C$, we already know that $A_1 A_2$ is smoothing in a conic neighborhood of  $\Sigma_l\cup\Sigma_t$, and that $C^{-1} \circ C$ is the identity on $\Sigma_s$. In particular, $C^{-1} \circ C$ is an embedded submanifold of $T^* (M \times M) \setminus 0$ in a neighborhood of $\text{ess-sup}(A_1 A_2)$, and closed as its subset. 

Let us now show that (\ref{clean_intersection_Tp}) holds for 
(\ref{splike_cover}). We write $\CC = \X \cap \Y \cap \Gamma_j$
and $\tilde C = C \cap \tilde \Gamma_j$,
where $\tilde \Gamma_j$ is the projection of $\Gamma_j$ on $T^* \mathcal M \times \Sigma_s$.
Also, we use $\cong$ to denote that two manifolds or vector spaces are isomorphic.  
Let 
$
(\tilde t, \tilde x, \tilde \tau, \tilde \xi; z, \theta, \zeta, \hat \theta; z, \theta, \zeta, \hat \theta;
t,x,\tau,\xi) \in \CC.
$
Since 
$$
(\tilde t, \tilde x, \tilde \tau, \tilde \xi) = C^{-1}(z, \theta, \zeta, \hat \theta) = (t,x,\tau,\xi)
$$
and $(\tau, \xi)$ is spacelike, Lemma \ref{lem_C_min_2} implies that $(\tilde t, \tilde x, \tilde \tau, \tilde \xi) = (t,x,\tau,\xi)$.
This again implies that $\CC \cong \diag(\tilde C) \cong \tilde C$, in particular, $\CC$ is a smooth manifold.
Moreover, $\X \cong C^2$.
Let $p \in \tilde C$ and observe that for $(\delta p, \delta q)$ in $T_{(p,p)} \X \cong T_p C \times T_p C$ it holds that $(\delta p, \delta q) \in T_{(p,p)} \Y$ if and only if 
$$
\d \pi_{\mathcal M} \delta p = \d \pi_{\mathcal M} \delta q.
$$
Since $\d \pi_{\mathcal M}$ is injective (again due to $(\tau, \xi)$ being spacelike), we have $\delta p = \delta q$ for all $(\delta p, \delta q)$ in $T_{(p,p)} \X \cap T_{(p,p)} \Y$. Therefore 
$$
T_{(p,p)} \X \cap T_{(p,p)} \Y \cong T_p \tilde C \cong T_{(p,p)} \CC.
$$
Keeping track of the diffeomorphisms used above, this shows (\ref{clean_intersection_Tp}). We have shown that the clean intersection calculus applies, and therefore $L_\kappa' Q L_\kappa$ is a pseudodifferential operator.

To establish that $L'_\kappa Q L_\kappa$ has order $-1$, we need to verify also that the order $m = -n/4$ of $L_\kappa$ and the excess $e$ of the clean intersection satisfy
$2m + e/2 = -1$. 
We write 
$$
\Pi : \CC \to T^* (M \times M) \setminus 0 
$$
for the natural projection, and $\CC_\gamma = \Pi^{-1}(\{\gamma\})$ for its fibers. The excess $e$ coincides with $\dim(\mathcal C_\gamma)$, and 
using again the identification $\CC \cong \tilde C$, we see that for all $\gamma \in \Pi(\CC)$ there is $(t,x,\tau,\xi) \in \Sigma_s$ such that 
$$
\CC_\gamma \cong 
\{(t,x,\tau,\xi; z, \theta, \zeta, \hat \theta);\ (z, \theta, \zeta, \hat \theta) \in C(t,x,\tau,\xi) \}
\cong C(t,x,\tau,\xi) \cong S^{n-2},
$$
where the last identification is given by part (b) of Lemma~\ref{lem_C_min_1}.
Hence $e = n-2$ and indeed $2m + e/2 = -1$. 

We remark that, in the context of \cite[Th. 25.2.3]{Hormander4}, the composition $C_1 \circ C_2$ being connected means that the fibers $\CC_\gamma$ are connected (when $\CC$ is taken to be the whole intersection $\X \cap \Y$). As we are assuming that $n > 2$, the fibers $\CC_\gamma$ are connected in our particular case. With a suitable cutoff, this can be arranged also in the more general Lorentzian context considered next, however, analogously to the above discussion, such connectedness is not essential. Even when not connected, the fibers $\CC_\gamma$ are smooth manifolds, since the projection $\Pi$ has constant rank by \cite[Th. 21.2.14]{Hormander3}.

\section{The Lorentzian case}
\label{sec_Lorentz}

Our aim is to prove an analogue of Theorem~\ref{th_rec_min} in a more general Lorentzian context. Toward this end, we will consider the light ray transform on a Lorentzian manifold $(M,g)$, localized near a lightlike geodesic segment $\gamma_0 : [0, \ell] \to M$, the analogue of $\ell_{z_0, \theta_0}$ in Theorem~\ref{th_rec_min}. 
We parameterize lightlike geodesics near $\gamma_0$
by choosing a spacelike hypersurface $H$ containing $\gamma_0(0)$ and semigeodesic coordinates
associated to $H$, 
    \begin{equation}\label{coords_tz}
(t,z) \in (-T, T) \times Z, \quad Z \subset \R^{n},
    \end{equation}
so that in the coordinates $H = \{t = 0\}$ and 
$g = -dt^2 + g'$, with $g'=g'(t,z)$ 
a Riemannian metric on $Z$ that depends smoothly on $t$. Moreover, the coordinates are chosen so that $\gamma_0(0) = 0$ and $\dot \gamma_0(0) = (1,\theta_0)$
where, writing $h = g'(0,\cdot)$, it holds that $(\theta_0,\theta_0)_{h} = 1$.
Then we choose local coordinates of the form 
    \begin{equation}\label{coords_za}
(z,a) \in Z \times A, \quad A \subset \R^{n-1},
    \end{equation}
on the unit sphere bundle $SZ$ with respect to $h$,
so that writing  
    \begin{align}\label{def_theta}
Z \times A \ni (z,a) \mapsto (z, \theta(z,a)) \in SZ
    \end{align}
for the coordinate map,
it holds that $\theta_0 = \theta(0,0)$. 
We write $\gamma_{z,a}$ for the geodesic $\gamma$ satisfying $\gamma(0) = (0,z)$ and $\dot \gamma(0) = (1,\theta(z,a))$, and use also the notation $\mathcal M = Z \times A$. 
Analogously to (\ref{g}), this parametrization gives the smooth manifold structure in the space of lightlike geodesics near $\gamma_0$. 

Let $\Omega\subset M$ be open and {\ntext relatively compact},
and suppose that the end points $\gamma_0(0)$ and $\gamma_0(\ell)$ are outside $\overline \Omega$. By making $Z$ and $A$ smaller, we suppose without loss of generality that the end points $\gamma_{z,a}(0)$ and $\gamma_{z,a}(\ell)$ are outside $\overline \Omega$ for all $(z,a) \in \overline{\mathcal M}$. 
In what follows we consider the local version of the light ray transform defined as follows
    \begin{align}\label{L_local}
L_\kappa f(\gamma) = \int_0^\ell \kappa(\gamma(s), \dot \gamma(s)) f(\gamma(s))\, \d s, \quad f \in C_0^\infty(\Omega),\ \gamma=\gamma_{z,a},\ (z,a) \in \mathcal M.
    \end{align}
Observe that, given a geodesic $\gamma : \R \to M$, the integral $\int_\R f(\gamma(s))\,\d s$ may not be well-defined even for $f \in C_0^\infty(\Omega)$ if $\gamma$ returns to $\Omega$ infinitely often. {\ntext We note that if $(M,g)$  is globally hyperbolic,  $L_\kappa f(\gamma)$ can be defined for
all  $f \in C_0^\infty(M)$. However, in this paper we} consider only the local version (\ref{L_local}) in order to avoid making global assumptions on $(M,g)$.

Note that the coordinates (\ref{coords_tz}) are valid locally only; and we cannot use them in our analysis of the contributions of possible conjugate points on the geodesics $\gamma_{z,a}$. They are used only to parametrize these geodesics.  
Moreover, the parametrization and, in particular, the normalization of $\dot \gamma(0)$, is not invariant. It depends on the choice of $H$ and the coordinates (\ref{coords_tz})--(\ref{coords_za}). On the other hand, if $\tilde H$ is another spacelike hypersurface intersecting $\gamma_0$, then the lightlike geodesic flow provides a natural map from $TH$ to $T\tilde H$, however, the projections of the tangents of the geodesics $\gamma_{z,a}$ onto $T\tilde H$ may not be of unit length. If the geodesics $\gamma_{z,a}$ are re-parameterized so that the projections are unit, then the weight $\kappa$ is multiplied by a smooth Jacobian. While this would change $L_\kappa$, it would not change its microlocal properties. We will use this fact later to choose $H$ in a convenient way.

\begin{figure}[!ht] 
  \centering
  \includegraphics[width = 2in,page=1]{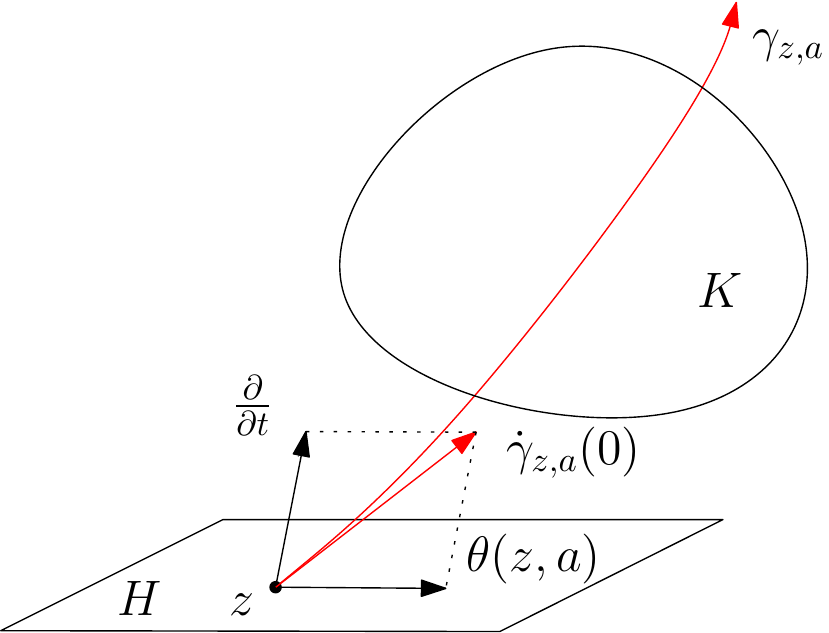}
  \caption{Parameterization of lightlike geodesics near $\gamma_0$.}
  \label{pic02}
\end{figure}
 
\subsection{Point-geodesic relation}

The point-geodesic relation
\be{Z}
X =  \left\{ (z,\a,x) \in \mathcal M \times \Omega|\; x=\gamma_{z,\a}(s)\;\text{for some $s \in (0,\ell)$} \right\}
\ee
is a smooth $2n$ dimensional submanifold  of the $3n$ dimensional $\mathcal M \times \Omega$, parameterized by 
the map $(z,\a,s)\mapsto(z,\a,  \gamma_{z,\a} (s) )$.
Writing $x = \gamma_{z,a}(s)$, this map  has differential
\[
 \begin{pmatrix} \Id & 0 &0\\0& \Id&0\\ 
\partial x/\partial z &\partial x/\partial \a & \dot \gamma_{z,\a} (s) 
\end{pmatrix}  \begin{pmatrix} \d z\\\d\a\\ \d s 
\end{pmatrix} 
\]
which has maximal rank $2n$. The conormal bundle $N^*X$ at any point is the space conormal to the range of that differential; that is, it is described by the kernel of its adjoint. Therefore, the canonical relation $C:=N^*X' \subset T^*(\mathcal M \times \Omega) \setminus 0$ is given by
\be{NZa}
\begin{split}
C &=  \Big\{  \big((z,\a,\zeta,\alpha) , (x, \xi )     \big)\big| \;  x=\gamma_{z,\a}(s),\,  \langle \xi, \dot\gamma_{z,\a}(s)\rangle =0,\,      \zeta_j=  \langle\xi,  \partial_{z^j}\gamma_{z,\a}(s)\rangle,      \\ &\qquad \; j=1,\dots,n, \, \alpha_k=  \langle\xi, \partial_{\a^k}\gamma_{z,\a}(s) \rangle,\ k=1,\dots,n-1,\ s \in (0,\ell) \Big\}.
\end{split}
\ee
Clearly $\zeta = 0$ and $\alpha = 0$ if $\xi = 0$.
It follows from Lemma \ref{lem_Jacobi} below that also the converse holds. Therefore $C$ is closed in $T^*(\mathcal M \times \Omega) \setminus 0$, and $L_\kappa$ is a Fourier integral operator.
The Schwartz kernel of $L_\kappa$ is a conormal distribution on $X$ with the (un-reduced) symbol $\kappa$, and by \cite[Th. 18.2.8]{Hormander3}, the order $m$ of $L_\kappa$ is satisfies  
$$
0 = m + (3n-2n)/4, \quad \text{that is} \quad m=-n/4.
$$

As in the Minkowski case, the covector $\xi$ must be lightlike or spacelike at $x$ as a consequence of $\langle \xi, \dot\gamma_{z,\a}(s)\rangle =0$. 
Relation \r{WF} holds  in this case as well and it shows that timelike singularities of $f$ do not affect 
$\WF(L_\kappa f)$, that is, they are invisible. 
Moreover, the dimensions of the manifolds in the diagram \r{4.2}  are unchanged from the Minkowski case.

The canonical relation $C$ is parameterized by 
\[
C_0=\{ (z,\a,s,\xi) \in \mathcal M \times (0,\ell) \times T^*_{\gamma_{z,a}(s)} \Omega\,|  \: \xi \ne 0, \langle \xi, \dot\gamma_{z,\a}(s)\rangle =0 \}.
\]
More precisely, $C_0$ is a $3n$-dimensional smooth manifold and, in view of the definition \r{NZa}, there is a diffeomorphism between $C_0$ and $C$.

\subsection{Variations of the geodesics $\gamma_{z,a}$}\label{sec_VG}
Let us consider the Jacobi fields associated to the variations through the geodesics $\gamma_{z,a}$, $(z,a) \in Z \times A$,
\be{MJ}
M_j(s;z,a)=  \partial_{z^j}\gamma_{z,\a}(s), \ j=1,\dots,n,\quad J_k(s;z,a) =  \partial_{a^k} \gamma_{z,\a} (s), \ k=1,\dots, n-1.
\ee
Observe that by \r{NZa}, it holds on $C$ that 
\be{Pi_M1}
\zeta_j=\langle\xi, M_j(s;z,a)\rangle,\quad\alpha_k= \langle\xi, J_k(s;z,a)\rangle,
\ee
that is, the dual variables $\zeta$ and $\alpha$  are given by projections of the Jacobi fields $M_j$ and $J_k$ to $\xi$. 

For a vector field $J$ along a curve $\gamma$,
we use the shorthand notation $ J'(s) = \nabla_s J(s)$ for  the covariant derivative $\nabla_s = \nabla_{\dot \gamma}$ along $\gamma$. We write also
$$
\dot\gamma(s)^\perp = \{v\in T_{\gamma(s)}M|\; (v,\dot\gamma(s))_g=0 \}.
$$
Since every Jacobi field along a null geodesic is a certain variation of the latter, the lemma below in particular characterizes the Cauchy data $(J,J')$ of such fields at any point. 

\begin{lemma}\label{lem_Jacobi}
Let $(z,a) \in \mathcal M$ and write $\gamma = \gamma_{z,a}$. Write $\Gamma(s) = s \dot\gamma(s)$, and 
consider the Jacobi fields 
$
\mathcal J := \spa \{ M_1, \dots, M_{n}, J_1, \dots, J_{n-1}, \dot \gamma, \Gamma\}
$
along $\gamma$. Then for any $s \in [0,\ell]$ it holds that 
$$
\{(J(s), J'(s))|\; J \in \mathcal J\}
= \{(V,W) \in (T_{\gamma(s)} M)^2 |\; W \in \dot\gamma(s)^\perp\}.
$$
In particular, $\{J(s)|\; J \in \mathcal J\} = T_{\gamma(s)} M$.
\end{lemma}
\begin{proof}
Let us begin by showing
that $( J'(0), \dot \gamma(0))_g = 0$ for $J \in \mathcal J$.
Consider the curve $r \mapsto (0, z + r e_j)$ in coordinates (\ref{coords_tz}), where $z \in Z$ is fixed and $e_j$ is the $n$-dimensional vector with $1$ in the $j$th position, all other entries zero, and denote by $\nabla_{z^j}$ the covariant derivative along this curve. Using the symmetry property $\nabla_s \p_{z^j} \gamma = \nabla_{z^j} \p_s \gamma$, we see that
\be{Jj}
\big(M_j(0), M_j'(0)\big)= ((0,e_j),(0 , \nabla_{z^j} \theta )), \quad \big(J_k(0), J_k'(0)\big)= \big((0,0),(0,\partial_{a^k} \theta)\big),
\ee
where $\theta$ is the map defined by (\ref{def_theta}).
Hence 
$( M'_j(0), \dot \gamma(0))_g
= (\nabla_{z^j} \theta, \theta)_h = \p_{z^j} (\theta, \theta)_h / 2= 0$.
Similarly also $(J'_k(0), \dot \gamma(0))_g = 0$.
Finally, as $\dot \gamma' = 0$ and $ \Gamma' = \dot \gamma$, we have shown that $( J'(0), \dot \gamma(0))_g = 0$ for $J \in \mathcal J$.

Recall that $\p_s ( J'(s), \dot \gamma(s))_g = 0$ for any Jacobi field $J$ along $\gamma$.
Therefore $(J', \dot \gamma)_g = 0$
identically on $[0,\ell]$ for $J \in \mathcal J$.
In particular,
    \begin{equation}\label{Js_incl}
\mathbb J(s) := \{(J(s), J'(s))|\; J \in \mathcal J\}
\subset \{(V,W) \in (T_{\gamma(s)} M)^2 |\; W \in \dot\gamma(s)^\perp\}.
    \end{equation}

The vectors $(J(0), J'(0))$, $J \in \mathcal J$, are linearly independent, as can be seen from (\ref{Jj}) and from 
\[
(\dot\gamma(0), \gamma''(0)) 
= ((1,\theta), 0), \quad
(\Gamma(0), \Gamma'(0)) 
= (0, (1,\theta)).
\]
As Jacobi fields satisfy a linear second order differential equation, it follows that the dimension of $\mathcal J$ is $2n+1$ and that the same is true for $\mathbb J(s)$, $s \in [0,\ell]$.
The claim follows from (\ref{Js_incl}) since both the spaces there have the same dimension.
\end{proof}

For fixed $s_1, s_2 \in [0,\ell]$,  consider the spaces 
    \begin{equation}\label{def_J0}
\mathcal J_{s_1} = \{J \in \mathcal J|\; J(s_1) = 0\},
\quad 
\mathcal J_{s_1,s_2}' = \{J'|\; J\in\mathcal{J}_{s_1} \cap \mathcal{J}_{s_2}\},
    \end{equation}
and set for every $s \in[0,\ell]$ 
    \begin{equation}\label{def_J0s}
\mathcal{J}_{s_1}(s) =  \{  J(s)|\;J\in\mathcal{J}_{s_1}  \}, \quad 
\mathcal{J}_{s_1, s_2}'(s) =  \{ J'(s)|\; J' \in \mathcal J_{s_1,s_2}'\}.
    \end{equation}
Lemma \ref{lem_Jacobi} implies that 
$\mathcal{J}'_{s_1,s_2}(s)\subset \dot\gamma(s)^\perp$.
The same is true for $\mathcal{J}_{s_1}(s)$ since
$$
\p_s (J, \dot \gamma)_g = (J', \dot \gamma)_g = 0, \quad J \in \mathcal J,
$$
and $(J(s_1), \dot \gamma(s_1))_g = 0$ for $J \in \mathcal{J}_{s_1}$.
To summarize, for every $s \in [0,\ell]$,
\be{J0}
\mathcal{J}_{s_1}(s)\cup \mathcal{J}'_{s_1,s_2}(s)\subset \dot\gamma(s)^\perp,
\ee
and in particular, both  spaces consist of spacelike or lightlike vectors. 
 Furthermore, 
 $\dot \gamma(s)\in \mathcal{J}_{s_1}(s)$ if and only $s\not=s_1$, because $(s-s_1)\dot\gamma(s)\in \mathcal J_{s_1}$. 
 On the other hand, 
 $\dot \gamma(s) \in \mathcal{J}_{s_1,s_2}'(s)$ if and only $s=s_1=s_2$.

We will need below the following simple lemma. 

\begin{lemma}\label{lem_perp_lightlike}
If two lightlike vectors $v, w \in T_x M$ satisfy $(v,w)_g = 0$ then they are parallel. 
\end{lemma}
\begin{proof}
 We can choose local coordinates so that $g$ coincides with the Minkowski metric at $x$. Then $v$ and $w$ are parallel with vectors of the form $(1,\theta)$ and $(1,\omega)$ with $\theta$ and $\omega$ unit vectors.
Now $(v,w)_g = 0$ implies that $\omega \cdot \theta = 1$, and thus $\omega$ and $\theta$ must be parallel.
\end{proof}

We will need the following property: for any Jacobi fields $I, J$ along a geodesic $\gamma$, the Wronskian 
\be{Wronskian}
(I,J')_g - ( I',J)_g\quad  \text{is  constant along  $\gamma$}, 
\ee
see e.g. 
\cite[p.~274]{ONeill-book}. 

\begin{lemma}\label{lemma_J} 
Let $(z,a) \in \mathcal M$ and write $\gamma = \gamma_{z,a}$. 
Then for every $s_1,s_2\in[0,\ell]$, we have 
\begin{itemize}
\item[(i)] $\mathcal{J}_{s_1}(s_2)$ and $\mathcal{J}_{s_1,s_2}'(s_2)$ are mutually orthogonal with respect to $g$,
\item[(ii)] $\mathcal{J}_{s_1}(s_2)\cap \mathcal{J}_{s_1,s_2}'(s_2)=\{0\}$,
\item[(iii)] $\mathcal{J}_{s_1}(s_2)+ \mathcal{J}_{s_1,s_2}'(s_2) = \dot \gamma(s_2)^\perp$.
\end{itemize}
\end{lemma}
\begin{proof}
Note first that if $s_2=s_1$, then $\mathcal{J}_{s_1}(s_2)= \{0\}$ and $\mathcal{J}_{s_1,s_2}'(s_2) = \dot\gamma(s_0)^\perp$ by Lemma~\ref{lem_Jacobi}. Therefore the lemma holds in this case, and we can assume $s_2\not=s_1$ in what follows. 

For $w\in T_{\gamma(s_2)}M$ with $w\in \mathcal{J}_{s_1,s_2}'(s_2)$, let $I\in \mathcal{J}_{s_1}$ be the Jacobi field with Cauchy data $(0,w)$ at $s=s_2$. (If $I\ne0$, then $\gamma(s_1)$ and $\gamma(s_2)$ are conjugate along $\gamma$.) 
By \r{Wronskian},  for every $J\in\mathcal{J}_{s_1}$, we get $(w, J(s_2))_g=( I'(s_1),J(s_1))_g - ( I(s_1),J'(s_1))_g=0$, therefore, $w$ is orthogonal to $\mathcal{J}_{s_1}(s_2)$. This proves (i). 

To prove (ii), assume that $w\in\mathcal{J}_{s_1}(s_2)\cap \mathcal{J}_{s_1,s_2}'(s_2)$. Then $w$ is orthogonal to itself by (i), therefore it is lightlike. By \r{J0} it is also perpendicular to the lightlike vector $\dot\gamma(s_2)$, and Lemma \ref{lem_perp_lightlike} implies that $w$ must be parallel to $\dot \gamma(s_2)$. That is, $w=\lambda\dot\gamma(s_2)$ with some $\lambda \in \R$. Since $w\in \mathcal{J}_{s_1}'(s_2)$, then there is $J\in\mathcal{J}_{s_1}$ with Cauchy data $(0,\lambda\dot\gamma(s_2))$ at $s=s_2$; but then $J(s)=\lambda(s-s_2)\dot\gamma(s)$. Now $J\in \mathcal{J}_{s_1}$ and $s_1 \ne s_2$ imply $\lambda=0$, hence $J=0$ and also $w = 0$.

Consider now (iii). We write $W = \mathcal{J}_{s_1}(s_2)+ \mathcal{J}_{s_1,s_2}'(s_2)$. As $W \subset \dot \gamma(s_2)^\perp$
by (\ref{J0}), it remains to show the opposite inclusion. We will establish this by showing that $W^\perp$ is contained in $\R \dot \gamma(s_2) := \{\lambda \dot \gamma(s_2)|\; \lambda \in \R\}$. Then 
$\dot \gamma(s_2)^\perp \subset (W^\perp)^\perp$
and (iii) follows from $(W^\perp)^\perp = W$,
see e.g. \cite[Lemma 22, p. 49]{ONeill-book} for the latter fact.

Let  $w\in W^\perp$ and let $I$ be the Jacobi field with Cauchy data $(0,w)$ at $s=s_2$.
As $w$ is in particular orthogonal to $\mathcal{J}_{s_1}(s_2)$, by using \r{Wronskian} we get for every $J\in \mathcal{J}_{s_1}$, 
$$
(I(s_1),J'(s_1))_g = (I(s_2),J'(s_2))_g - (I'(s_2),J(s_2))_g =0. 
$$
Recall that by Lemma~\ref{lem_Jacobi}, 
$\{J'(s_1)|\; J\in\mathcal{J}_{s_1} \} = \mathcal{J}'_{s_1,s_1}(s_1)= \dot\gamma(s_0)^\perp$.
Therefore $I(s_1)$ is in $(\dot\gamma(s_1)^\perp)^\perp = \R \dot\gamma(s_1)$ and we write $I(s_1) = \lambda \dot\gamma(s_1)$.
Then for the Jacobi field 
\[
K(s)=I(s) + \lambda \frac{s-s_2}{s_2-s_1}\dot\gamma(s)
\]
it holds that $K(s_1)=0$ and $K(s_2)=0$. Writing
$u = K'(s_2)$ and $\mu = \lambda(s-s_0)^{-1}$, we have $u \in \mathcal{J}'_{s_1,s_2}(s_2)$ and
$u = w + \mu \dot \gamma(s)$.

Let us now use the fact that $w$ is orthogonal to the whole $W$.
It follows from (\ref{J0}) that $\R \dot \gamma(s_2) \subset W^\perp$ and therefore also $u = w + \mu \dot \gamma(s) \in W^\perp$. But $u \in \mathcal{J}'_{s_1,s_2}(s_2) \subset W$, and $u$ must be lightlike. 
Lemma \ref{lem_Jacobi} implies that $(u, \dot \gamma(s))_g = 0$ and then $u \in \R \dot \gamma(s)$ by Lemma \ref{lem_perp_lightlike}.
Hence also $w \in \R \dot \gamma(s)$.
\end{proof}

We will denote  by $\zeta_* \in T_z Z$ the image of $\zeta \in T_z^* Z$, with $z \in Z$, under the canonical isomorphism induced by $h$, i.e., $\zeta_*^j = h^{jk} \zeta_k$. Analogously for $\xi \in T_x^* M$, with $x \in M$, we denote by $\xi_* \in T_x M$  the vector defined by $\xi_*^j = g^{jk} \xi_k$.

Recall that in the Minkowski case the lightlike covectors on the canonical relation are characterized by (\ref{minkowski_lightlike}), or equivalently by $\xi \parallel \theta$. These two characterizations have the following analogues in the present context.

\begin{lemma}\label{lem_lightlike}
Let $(z,a,s,\xi) \in C_0$. Then  the following three conditions are equivalent:
\begin{itemize}
\item[(i)] $\xi$ is lightlike,
\item[(ii)] $\xi_*$ is parallel to $\dot \gamma_{z,a}(s)$,
\item[(iii)] $\zeta_*$ is parallel to $\theta(z,a)$
and $\alpha = 0$
where $\zeta$ and $\alpha$ are given by (\ref{Pi_M1}).
\end{itemize}
\end{lemma}
\begin{proof}
We will suppress $(z,a)$ in the notation below. 
Let us suppose first that $\xi_*$ is lightlike and show that $\xi_*$ is parallel to $\dot \gamma(s)$. As $(\xi_*, \dot \gamma(s))_g = \pair{\xi,\dot \gamma(s)} = 0$, Lemma \ref{lem_perp_lightlike} implies that $\xi_*$ is parallel to $\dot \gamma(s)$. 

Let us now suppose that $\xi_* = \lambda \dot \gamma(s)$ for some $\lambda \in \R$, and show that $\zeta_* = \lambda \theta$ and $\alpha = 0$.
Lemma \ref{lem_Jacobi} implies 
$
\p_s (M_j(s), \dot \gamma(s))_g = (M_j'(s), \dot \gamma(s))_g = 0.
$
Hence using also (\ref{Jj})
    \begin{equation*}
\zeta_j = \lambda (\dot \gamma(s), M_j(s))_g
= \lambda (\dot \gamma(0), M_j(0))_g
= \lambda ((1,\theta),(0,e_j))_g = \lambda \theta^{k} h_{kj}. 
    \end{equation*}
This establishes $\zeta_* = \lambda \theta$.
Analogously, 
$
\alpha_k = \lambda (\dot \gamma(0), J_k(0))_g
= 0
$ since $J_k(0) = 0$.

Let us now suppose that $\zeta_* = \lambda \theta$ and $\alpha = 0$ and show that $\xi_* = \lambda \dot \gamma(s)$.
The equations in the previous step imply that 
$$
(\xi_*, M_j(s))_g = \zeta_j = \lambda (\dot \gamma(s), M_j(s))_g, \quad
(\xi_*, J_k(s))_g = \alpha_k = 0 = \lambda (\dot \gamma(0), J_k(0))_g.
$$ 
Moreover, $(\xi_*, \dot \gamma(s))_g = 0 = \lambda(\dot \gamma(s), \dot \gamma(s))_g$. By Lemma \ref{lem_Jacobi}, $\{J(s)|\; J \in \mathcal J\} = T_{\gamma(s)} M$, and hence $\xi_* = \lambda \dot \gamma(s)$. This again implies that $\xi_*$ is lightlike.
\end{proof}

\subsection{The projection $\pi_\mathcal{M}$} 
We analyze $\pi_\mathcal{M}$ next.  We have  
\be{pi_calMa}
\pi_\mathcal{M}  \big((z,\a,\zeta,\alpha) , (x, \xi )     \big)  =    (z,\a,\zeta,\alpha) .
\ee
Since $C$ is parameterized by $(z,\a,s,\xi)\in C_0$, we view $\pi_\mathcal{M}$ as a function of those parameters.

As before, this projection is a map from the $3n$ dimensional $C$ to the $4n-2$ dimensional $T^*\mathcal{M}$. 
To see whether $\pi_\mathcal{M}$ is injective, let the right-hand side of \r{pi_calMa} be given. This means in particular that the geodesic $\gamma_{z,\a}$ is fixed. We want to find out whether the defining equations of $C$, that is, 
    \begin{equation}\label{eq_inj}
\zeta_j=\langle\xi, M_j(s;z,a)\rangle,\quad\alpha_k= \langle\xi, J_k(s;z,a)\rangle, \quad \langle \xi, \dot\gamma_{z,\a}(s)\rangle =0,
    \end{equation}
have more than one solution for $s$ and $\xi$.  

\begin{lemma}\label{lem_form}
Let $(z,a) \in \mathcal M$, $s_1 \ne s_2$, $J' \in \mathcal J'_{s_1,s_2}$, and let $\lambda \in \R$. Then 
$$
\xi_*^j = J'(s_j) + \lambda \dot \gamma(s_j), \quad j=1,2,
$$ 
satisfy
    \begin{equation}\label{ZA_non_inj}
\langle\xi^1, M_j(s_1;z,a)\rangle
=\langle\xi^2, M_j(s_2;z,a)\rangle
\quad 
\langle\xi^1, J_k(s_1;z,a)\rangle
= \langle\xi^2, J_k(s_2;z,a)\rangle,
    \end{equation}
and $\langle \xi^j, \dot\gamma_{z,\a}(s_j)\rangle =0$, $j=1,2$.
\end{lemma}
\begin{proof}
The claimed equations are linear, so it is enough to verify that the choices $\xi_*^j = J'(s_j)$ and $\xi_*^j = \dot \gamma(s_j)$ satisfy them. We begin with the former choice. 
By (\ref{Wronskian}) it holds that 
    \begin{align*}
\pair{\xi^1, M_j(s_1)} 
&= 
(J'(s_1), M_j(s_1))_g - (J(s_1), M_j'(s_1))_g
\\&= 
(J'(s_2), M_j(s_2))_g - (J(s_2), M_j'(s_2))_g
= \pair{\xi^2, M_j(s_2)},
    \end{align*}
and analogously $\pair{\xi^1, J_k(s_1)} 
= \pair{\xi^2, J_k(s_2)}$. The last equation follows from (\ref{J0}).
Let us now consider the choice $\xi_*^j = \dot \gamma(s_j)$.
By Lemma \ref{lem_Jacobi} the scalar products
$(M_j', \dot \gamma)_g$ and $(J_k', \dot \gamma)_g$
and vanish identically. Thus $(M_j, \dot \gamma)_g$ is constant along $\gamma$, and the same holds for $(J_k, \dot \gamma)_g$. Therefore $\dot \gamma(s_1)$ and $\dot \gamma(s_2)$ solve (\ref{ZA_non_inj}).
The last equation holds since $\gamma$ is lightlike.
\end{proof}

\begin{lemma}\label{lem_conj}
Let $(z,a,\zeta,\alpha) \in T^* \mathcal M$ and let $(s_j,\xi^j) \in [0,\ell] \times T_{\gamma_{z,a}(s_j)}^* M$, 
$j=1,2$, solve (\ref{eq_inj}). Then the following hold:
\begin{itemize}
\item[(i)] Either both $\xi^1$ and $\xi^2$ are spacelike or they are both lightlike.
\item[(ii)] If $s_1 = s_2$ then $\xi^1 = \xi^2$.
\item[(iii)] If $s_1 \ne s_2$ then 
there are unique $J' \in \mathcal J'_{s_1,s_2}$ and $\lambda \in \R$ such that 
$$
\xi_*^j = J'(s_j) + \lambda \dot \gamma(s_j), \quad j=1,2.
$$
Moreover, $\xi^1$ and $\xi^2$ are spacelike if and only if $J' \ne 0$.
\end{itemize}
\end{lemma}

Let us remark that the case $J'=0$ in (iii) is the analogue of the
fact that in the Minkowski case, equation (\ref{pi_calM_lightlike}) for lightlike $(\tau^j,\xi^j)$ is equivalent with $( \tau^1,\xi^1) = ( \tau^2,\xi^2)$ and both $(t_j, x_j)$, $j=1,2$, lying on the same line $\ell_{z,\theta}$.

\begin{proof}
We will again suppress $(z,a)$ in the notation below. 
We will begin by proving (i). Recall that $\langle \xi^j, \dot\gamma(s_j)\rangle =0$ implies that $\xi^j$ is lightlike or spacelike.  
It is enough to show that $\xi^1$ being lightlike implies that also $\xi^2$ is lightlike. So suppose that $\xi^1$ is lightlike. Then Lemma \ref{lem_lightlike} implies that $\zeta_*$ is parallel to $\theta$ and $\alpha = 0$. Therefore $\xi^2$ is lightlike by the same lemma.

Let us now show (ii). When $s_1 = s_2$, equation (\ref{eq_inj}) implies that 
$$
\pair{\xi^1,M_j(s_1)} = \pair{\xi^2,M_j(s_1)},
\quad \pair{\xi^1,J_k(s_1)} = \pair{\xi^2,J_k(s_1)},
\quad \pair{\xi^1,\dot \gamma(s_1)} = \pair{\xi^2,\dot \gamma(s_1)},
$$
and, as $\{J(s_1)|\; J \in \mathcal J\} = T_{\gamma(s_1)} M$ by Lemma \ref{lem_Jacobi}, it holds that $\xi^1 = \xi^2$.

We turn to (iii).
As $\xi^2_* \in \dot\gamma(s_2)^\perp$,
there are unique $u \in \mathcal{J}_{s_1}(s_2)$ and $w \in \mathcal{J}'_{s_1,s_2}(s_2)$
such that $\xi^2_* = u + w$ by Lemma \ref{lemma_J}.
As $(s_j,\xi^j)$ solve (\ref{eq_inj}), it holds that for all $J \in \mathcal J_{s_1}$ that
$$
0 = (\xi_*^1, J(s_1))_g = (\xi_*^2, J(s_2))_g.
$$
In other words, $\xi_*^2 \in \mathcal{J}_{s_1}(s_2)^\perp$.
By (i) of Lemma \ref{lemma_J}, also $w \in \mathcal{J}_{s_1}(s_2)^\perp$.
Therefore
$$
u = \xi_*^2 - w \in \mathcal{J}_{s_1}(s_2) \cap \mathcal{J}_{s_1}(s_2)^\perp
$$ 
and $u$ must be lightlike. 
As $u$ is orthogonal to $\dot \gamma(s_2)$ by (\ref{J0}), it follows from Lemma \ref{lem_perp_lightlike} that $u = \lambda \dot \gamma(s_2)$ for some $\lambda \in \R$.
Let $J$ be the Jacobi field with Cauchy data $(0,w)$ at $s = s_2$. Then $J(s_1) = 0$ since $w \in \mathcal{J}_{s_1,s_2}'(s_2)$.
Setting $\tilde \xi_*^{1} = J'(s_1) + \lambda \dot \gamma(s_1)$, the covectors $\tilde \xi^{1}$ and $\xi^2$ give a solution to (\ref{ZA_non_inj}) by Lemma \ref{lem_form}. 
It then follows from part (ii) that $\xi^1 = \tilde \xi^{1}$.

Clearly both $\xi^j$, $j=1,2$, are lightlike if $J'=0$.
On the other hand, if $\xi^j$, $j=1,2$ are lightlike, then $\xi^j_* \in \dot \gamma(s_j)^\perp$, applying Lemma~\ref{lem_perp_lightlike}, implies that $J'(s_j) = \mu \dot \gamma(s_j)$ for some $\mu \in \R$. 
Now $J(s_1) = 0$ and $J'(s_1) = \mu \dot \gamma(s_1)$ imply that $J(s)=\mu (s-s_1)\dot \gamma(s)$, and $J(s_2) = 0$ implies that $\mu = 0$.
\end{proof}

The above lemma says in particular that if there are two distinct solutions $(s_j,\xi^j)$, $j=1,2$, to (\ref{eq_inj}) and if $\xi^1$ is spacelike then $\gamma_{z,a}(s_1)$ and $\gamma_{z,a}(s_2)$
are conjugate along $\gamma_{z,a}$.
By Lemma~\ref{lem_form} the converse holds as well. Indeed, if 
$\gamma_{z,a}(s_1)$ and $\gamma_{z,a}(s_2)$
are conjugate along $\gamma_{z,a}$ then there is non-zero $J' \in \mathcal J'_{s_1,s_2}$ and for any $\lambda \in \R$ the vectors $\xi_*^j$ in Lemma~\ref{lem_form} are spacelike solutions to (\ref{ZA_non_inj}).
 
The characterization of the pairs $(\xi_1,\xi_2)$ is related to that in the Riemannian case, see  \cite[Theorem~4.2]{SU-caustics} where the conjugate points are assumed to be of fold type; see also \cite{Sean-X} for a more general case.

We will finish our study of $\pi_{\mathcal M}$ by showing that $\d \pi_{\mathcal M}$ is injective in the spacelike cone.

\begin{lemma}\label{lem_pimathcalM_inj}
Let $p_0 := ((z_0,a_0, \zeta^0,\alpha^0), (x_0,\xi^0))\in C$ and suppose that $\xi^0$ is spacelike. Then $\d\pi_\mathcal{M}$ is injective at $p_0$.
\end{lemma}
\begin{proof}
After reparametrization, we can assume $x_0\in H$, $x_0 = (0, z_0) = (0,0)$ in the semigeodesic coordinates (\ref{coords_tz}), and $a_0 = 0$.
In particular, we can consider $x = (x^0,x')$ near $x_0$ as a point in $(-T,T) \times Z$. 
We write also $\xi^0 = (\xi_0^0, {\xi^0}')$.
Then the points $(z,a,s,\xi)$ in $C_0$ near $(0,\xi^0)$ can be parameterized by $(z,a,s,\xi') \in Z \times A \times \R \times \R^n$ by setting $\xi = (\xi_0(z,a,s,\xi'), \xi')$ where $\xi_0(z,a,s,\xi')$ is the unique solution 
to 
$\pair{\xi, \dot \gamma_{z,a}(s)} = 0$
near $\xi_0^0$.
Indeed, this follows from the implicit function theorem since $\p_{\xi_0} \pair{\xi, \dot \gamma_{0}(0)} = 1$.
 
Using the above parameterization, we write $\pi_{\mathcal M}(z,a,s,\xi') = (z,a,\zeta,\alpha)$
with $\zeta$ and $\alpha$ as in (\ref{Pi_M1}).
To show that $\d\pi_{\mathcal{M}}$ is injective at $p_0$, it is enough to show that $\partial(\zeta,\alpha)/\partial(s,\xi')$ is injective at $(0, {\xi^0}')$. Moreover, using (\ref{Jj}), we have at $(0, {\xi^0}')$,
$$
\frac{\p \zeta}{\p \xi'} 
= \begin{pmatrix} M_1^1(0) &\cdots&M_n^1(0)\\ 
 \vdots &\ddots&\vdots\\
M_n^1(0)&\cdots &M_n^n(0)
\end{pmatrix}  
= \Id.
$$
As also $\p \alpha / \p \xi'  = 0$ there, it is enough to show that $\p \alpha / \p s  \ne 0$. Using once again (\ref{Jj}), it holds at $(0, {\xi^0}')$ that
$$
\frac{\p \alpha_k}{\p s} 
= (\nabla_s \xi_*, J_k)_g + (\xi_*, \nabla_s J_k)_g
= ({\xi_*^0}', \p_{a^k} \theta)_h.
$$

To get a contradiction, suppose that $({\xi_*^0}', \p_{a^k} \theta)_h = 0$, $k=1,\dots,n-1$. 
As the vectors $\p_{a^k} \theta$, $k=1,\dots,n-1$, span the tangent space of the unit sphere $S_{z_0} Z$ at $\theta_0$, the vector ${\xi_*^0}'$
must be parallel to $\theta_0$. But then $\pair{\xi^0,\dot \gamma_0(0)} = 0$ implies that $\xi_*^0$ is parallel to $(1,\theta_0)$, a contradiction with $\xi^0$ being spacelike. 
\end{proof}

\subsection{The projection $\pi_{M}$} 
As above, we regard the projection $\pi_M$ in \r{4.2} as a map of $C$ parameterized by $(z,a,s,\xi)\in C_0$ to $T^*M$. We have
\[
\pi_{M} \big((z,\a,\zeta,\alpha) , (x, \xi )     \big) =     (x,\xi )  = ( \gamma_{z,\a}(s) , \xi),
\]
with $\xi$ conormal to $\dot \gamma_{z,\a}(s)$. It maps the $3n$ dimensional $C$ to the $2n+2$ dimensional $T^*M$.  
Moreover, $\pi_M$ is surjective in the sense that 
there are $(z,a,s) \in Z \times A \times (0,\ell)$ satisfying 
    \begin{equation}\label{piM_eqs}
x=\gamma_{z,a}(s), \quad \pair{\xi, \dot\gamma_{z,a}(s)} = 0,
    \end{equation}
assuming that $(x,\xi) \in \Sigma_s$ is close to $N^* \gamma_0$. 
Indeed, as in the Minkowski case, 
solving for $\eta_* = \dot\gamma_{z,a}(s)$ modulo rescaling in the second equation in (\ref{piM_eqs}), we obtain
a $(n-2)$-dimensional sphere of lightlike solutions  when $\xi$ is spacelike; and two distinct   vectors when $n=2$. 
Moreover, when $x$ is close to $\gamma_0(s_0)$ 
for some $s_0 \in (0,\ell)$ and $\xi$
is close to $N_{\gamma_0(s_0)}^* \gamma_0$, we can choose $\eta_*$ near $\dot \gamma_0(s_0)$. Then finding $z$ and $a$ is straightforward because $H$ is transversal to $\gamma_0$. 

It follows from \cite[Prop. 25.3.7]{Hormander4}
that the differential $\d\pi_M$ is surjective whenever $\d \pi_{\mathcal M}$ is injective. Let us, however, show this also directly for a point $p_0$ as in Lemma \ref{lem_pimathcalM_inj}. We re-parametrize again as in Lemma \ref{lem_pimathcalM_inj}.
Then at $(0, {\xi^0}')$
$$
\d\pi_M = \frac{\p (x^0, x', \xi_0, \xi')}
{\p (z,a,s,\xi')}
= \begin{pmatrix}
0 & 0 & 1 & 0
\\
\Id & 0 & \p x'/\p s & 0
\\
\p \xi_0 / \p z & \p \xi_0 / \p a & \p \xi_0 / \p s & \p \xi_0 / \p \xi'
\\
0 & 0 & 0 & \Id
\end{pmatrix},
$$
and we see that $\d\pi_M$ is surjective if and only if $\p \xi_0 / \p a \ne 0$.
It follows from $\pair{\xi, \dot \gamma_{z,a}(s)} = 0$ and (\ref{Jj}) that
$$
0 = \frac{\p \xi_0}{\p a^k} + \pair{\xi^0, \p_{a^k} \dot \gamma_0} 
= \frac{\p \xi_0}{\p a^k} +  ({\xi_*^0}', \p_{a^k} \theta)_h.
$$
We showed in Lemma \ref{lem_pimathcalM_inj} that $({\xi_*^0}', \p_{a^k} \theta)_h$ can not vanish for all $k=1,\dots,n-1$ when $\xi^0$ is spacelike. 
Thus $\d\pi_M$ is surjective in this case.

\subsection{Conclusions} 

Analogously to Lemma \ref{lem_pr_min}, we summarize the results above:

\begin{lemma}\label{lemma_pr3} 
The differential $\d\pi_\mathcal{M}$ is injective and the differential $\d \pi_M$ is surjective at $(z,a,s,\xi) \in C_0$, with $\xi$ spacelike.
The projection $\pi_{\mathcal M}$ is injective in
a neighborhood of the set of points 
$(0,0,s,\xi) \in C_0$, with $\xi$ spacelike, if and only if there are no conjugate points on $\gamma_0$. 
The projection $\pi_M$ is surjective onto a neighborhood of $\Sigma_s \cap N^* \gamma_0$ in $\Sigma_s$.
\end{lemma}

We have also the following partial analogues of
Lemmas \ref{lem_C_min_1} and \ref{lem_C_min_2}, where write again $C = \pi_{\mathcal M} \circ \pi_M^{-1}$.

\begin{lemma}\label{lem_C_1}
For all $(x,\xi)$ in a small enough neighborhood of $\Sigma_s \cap N^* \gamma_0$ in $\Sigma_s$ it holds that $C(x,\xi)$ is the $(n-2)$-dimensional manifold given by 
$$\{(z,a,\zeta,\alpha) \in T^* \mathcal M\, |\, \text{(\ref{piM_eqs}) and (\ref{Pi_M1}) hold for some $s \in (0,\ell)$}\}.$$
\end{lemma}
\begin{proof}
If $((z,a,\zeta,\alpha),(x,\xi)) \in \pi_M^{-1}(x,\xi)$, then $(z,a)$ satisfies (\ref{piM_eqs}) for some $s \in (0,\ell)$. By the argument above, the solutions to this equation form a $(n-2)$-dimensional manifold. For each solution $(z,a)$, the parameter $s$ is fixed by $x = \gamma_{z,a}(s)$, and then $\zeta$ and $\alpha$ are given by (\ref{Pi_M1}).
\end{proof}

\begin{lemma}\label{lem_C_2}
Suppose that $(z,a,\zeta,\alpha) \in T^* \mathcal M$ is in the domain of $C^{-1}$ and not in the set
    \begin{align}\label{dual_L}
\mathcal L = \{(z,a,\zeta,\alpha) \in T^* \mathcal M|\; \zeta_*||\theta(z,a),\ \alpha = 0\}.
    \end{align}
Suppose, furthermore, that there are no conjugate points on $\gamma_0$. 
Then $C^{-1}(z,a,\zeta,\alpha) = (\gamma_{z,a}(s), \xi)$ where $(s,\xi)$ is the unique solution of (\ref{eq_inj}).
\end{lemma}
\begin{proof}
If $((z,a,\zeta,\alpha),(x,\xi)) \in \pi_{\mathcal M}^{-1}(z,a,\zeta,\alpha)$, then $\xi$ is spacelike by Lemma \ref{lem_lightlike}.
It follows from Lemma \ref{lem_conj} that (\ref{eq_inj}) has a unique solution $(s,\xi)$. Finally $x = \gamma_{z,a}(s)$ by (\ref{NZa}). 
\end{proof}

The analogue of Lemma~\ref{lem_CC_comp_min} reads:

\begin{lemma}\label{lem_CC_comp}
Suppose that there are no conjugate points on $\gamma_0$. 
For all $(x,\xi)$ in a small enough neighborhood of $\Sigma_s \cap N^* \gamma_0$ in $\Sigma_s$ it holds that $(C^{-1} \circ C)(x,\xi) = (x,\xi)$.
\end{lemma}
\begin{proof}
By Lemma~\ref{lemma_pr3}, the projection $\pi_{\mathcal M}$ is injective near the non-empty set $\pi_M^{-1}(x,\xi)$.
Therefore 
$
(C^{-1} \circ C)(x,\xi) = (\pi_M \circ \pi_M^{-1})(x,\xi) = (x,\xi).
$
\end{proof}

For a set $\mathcal V \subset T^*M \setminus 0$ we denote by $\mathcal V_c$ the conical set generated by $\mathcal V$, that is,
$$
\mathcal V_c = \{(x,\lambda\xi) \in T^*M \setminus 0 |\; (x,\xi) \in \mathcal V,\ \lambda > 0\}.
$$
Similarly to Theorem~\ref{th_rec_min}, we have:

\begin{theorem} \label{th_rec}
Suppose that there are no conjugate points on $\gamma_0$. Then 
there is a zeroth order pseudodifferential operator $\chi$ on $Z \times A$
 such that $L_\kappa' \chi L_\kappa$ is a pseudodifferential operator of order $-1$ 
with essential support in the spacelike cone.

Suppose, moreover, that $\kappa$ is nowhere vanishing. Then for any $(x_0,\xi^0) \in \Sigma_s \cap N^* \gamma_0$ 
the operator $\chi$ can be chosen so that $L_\kappa' \chi L_\kappa$ is elliptic at $(x_0,\xi^0)$.
\end{theorem}
\begin{proof}
Let $(x_0,\xi^0) \in T^*\Omega \cap \Sigma_s \cap N^* \gamma_0$
and let $s_0 \in (0,\ell)$ satisfy $x = \gamma_0(s_0)$.
Writing $z_0 = 0$ and $a_0 = 0$ we have $(z_0, a_0, s_0, \xi^0) \in C_0$. We define also $\zeta^0 = \zeta$ and $\alpha^0 = \alpha$ 
where $\zeta$ and $\alpha$ are given by (\ref{Pi_M1}) with $\xi = \xi_0$, $s = s_0$, $z = z_0$ and $a = a_0$.

Lemma \ref{lem_lightlike} implies that $(z_0,a_0,\zeta^0,\alpha^0)$ is outside the set 
$\mathcal L$ defined by (\ref{dual_L}).
We choose a neighborhood $\mathcal U \subset T^* \mathcal M$ of $(z_0,a_0,\zeta^0,\alpha^0)$ such that $\overline{\mathcal U}$ is compact and $\overline{\mathcal U} \cap \mathcal L = \emptyset$.
Moreover, we choose $\chi$ so that $\chi = 1$ near $(z_0,a_0,\zeta^0,\alpha^0)$ and so that it is essentially supported in $\mathcal U_c$.

The closed set $\pi_M^{-1}(\Sigma_l)$ is disjoint from the closed set
$\pi_{\mathcal M}^{-1}(\overline{\mathcal U})$ by Lemma \ref{lem_lightlike}.
We will show next that there is a conical neighborhood $W$ of $\pi_M^{-1}(\Sigma_l)$ such that $W \cap \pi_{\mathcal M}^{-1}(\overline{\mathcal U}) = \emptyset$.
It is enough to show that $\pi_{\mathcal M}^{-1}(\overline{\mathcal U})$ is bounded. This boils down showing that there is $C > 0$ such that all $(z,a,s,\xi) \in \pi_{\mathcal M}^{-1}(\overline{\mathcal U})$ satisfy $|\xi| \le C$. Consider the map $F$ taking $(z,a,s,\xi)$ to the point in $\R^{2n}$ with the coordinates
$$
\pair{\xi, M_1(z,a,s)},\dots,\pair{\xi, M_n(z,a,s)},\pair{\xi, J_1(z,a,s)},\dots,\pair{\xi, J_{n-1}(z,a,s)},\pair{\xi,\dot \gamma_{z,a}(s)}.
$$
Clearly $F$ is homogeneous of degree one in $\xi$, and by (\ref{NZa}),
$$
F(\pi_{\mathcal M}^{-1}(\overline{\mathcal U})) = \{ (\zeta,\alpha,0)\,|\, \text{$(z,a,\zeta,\alpha) \in \overline{\mathcal U}$ for some $(z,a)$}\}.
$$
But this set is bounded due to $\overline{\mathcal U}$ being compact. Therefore also $\pi_{\mathcal M}^{-1}(\overline{\mathcal U})$ is bounded.

As $\d \pi_M$ is surjective, $\pi_M$ is an open map and $\pi_M(W)$ is a neighborhood of $\Sigma_l$, considered as a subset of the range $\pi_M(C_0)$. We may choose a pseudodifferential operator $\tilde \chi$ so that $\tilde \chi = 1$ near $\Sigma_l$ and that is essentially supported in $\pi_M(W)$.
Then $\chi L_\kappa (1-\tilde \chi) = \chi L_\kappa$ modulo a smoothing operator. Moreover, $L_\kappa (1-\tilde \chi)$ is smoothing on $\Sigma_l$.

We can now apply the clean intersection calculus: the proof that (\ref{clean_intersection_Tp}) holds for (\ref{splike_cover})
is in verbatim the same as in the Minkowski case, except that we invoke Lemma \ref{lem_C_2} instead of Lemma \ref{lem_C_min_2}. Also $C^{-1} \circ C$ has the same global structure. Furthermore, the order is computed as in the Minkowski case, except that Lemma \ref{lem_C_1} is used instead of Lemma \ref{lem_C_min_1}.

For the claimed ellipticity, we choose $\chi$ so that $\sigma[\chi]$ is non-negative.
Note that the point $(z_0, a_0, \zeta^0, \alpha^0)$ is on the fiber $C(x_0,\xi^0)$. As $\chi = 1$ near $(z_0,a_0,\zeta^0,\alpha^0)$, the integral of $\sigma[\chi]$ does not vanish over the fiber $C(x_0,\xi^0)$. The ellipticity follows again from \cite[Theorem~25.2.3]{Hormander3}.
\end{proof}

Examples of metrics which do not allow conjugate points along lightlike geodesics include the Minkowski metric, product type of metrics $-\d t^2+g(x)$ with $g$ having no conjugate points, the Friedmann-Lema\^{\i}tre-Robertson-Walker (FLRW) metric $-\d t^2+a^2(t)\d x^2$ with $a>0$, and in particular the Einstein-de Sitter metric corresponding to $a(t)=t^{2/3}$; as well as metric conformal to them and small perturbations of all those examples  on compact manifolds. Of course, any Lorentizan metric is free of conjugate points on small enough subset of $M$. We refer to \cite{LOSU-strings} for the conformal invariance of this problem: the FLRW metric can be transformed into $a^2(s)(-\d s^2+\d x^2)$ after a  change of variables $s=s(t)$ solving $\d s/\d t=a^{-1}(t)$. Next two metrics conformal to each other have the same lightlike geodesics as smooth curves, but possibly parameterized differently, which does not change the property of existence or not of conjugate points. 
Going back to the original parameterization would multiply the weight $\kappa$ by a smooth non-vanishing factor, which would not change our conclusions.

\section{Cancellation of singularities in two dimensions} \label{sec_2D}
{\ntext Non-detectability and invisibility results have been extensively studied for inverse problems, 
see \cite{GKLU1,GKLU2,GLU1,GLU2} and references therein.}
For the Riemannian geodesic ray transform, it was shown in \cite{MonardSU14}, see also \cite{HMS-attenuated},  that in presence of conjugate points, singularities cannot be resolved locally, at least, i.e., knowing the ray transform near a single (directed) geodesic. We will prove an analogous result in the Lorentzian case in $1+2$ dimensions. 

We will review some of the results in section~\ref{sec_Lorentz} emphasizing on the specifics for the $n=2$ case. The point-geodesic relation $X$, see \r{Z}, is $4$-dimensional, and all manifolds in the diagram \r{4.2} (valid in the variable curvature case as well) are $6$ dimensional. The projection $\pi_\mathcal{M}$ is a local diffeomorphism in a neighborhood of a point $(\gamma_0,\hat\gamma^0,x_0,\xi^0)\in C$  with $(x_0,\xi^0)$ spacelike (here, $\hat\gamma$ is a dual variable to $\gamma = (z,a)$),   if and only if there are no points on $\gamma_0$ conjugate to $x_0$. The projection $\pi_M$ is also a local diffeomorphism under the same non-conjugacy condition. 
As a result, the canonical relation $C=  \pi_{\mathcal M}\circ \pi_M^{-1}$ is a local diffeomorphism from $\Sigma_s$ to its image. The composition as in Theorem~\ref{th_rec_min} then follows without the need to invoke the clean intersection calculus.  

We take a closer look at the geometry of the conjugate points when $n=2$. Two points along a geodesic are conjugate when there exists a non-zero Jacobi field vanishing at those points. This property is invariant under rescaling and shifting of the parameter $s$ of $\gamma(s)$, so we can take $s_1=0$. 
 A basis for $\mathcal{J}_{0}$ (the Jacobi fields vanishing at $0$, see \r{def_J0}), in local coordinates, is given by $J_1$, see \r{MJ}, and $s\dot \gamma(s)$. Since the second one does not vanish at $s\not=0$, a conjugate point could be at most at of order $1$, that is, the Jacobi fields $J$ with $J(0)=J(1)=0$ form an one-dimensional linear space, also true pointwise. One the other hand, at any point $\gamma(s)$, the conormal bundle to $\gamma$ is two-dimensional; and this is true for its restriction to the spacelike cone as well.

\begin{proposition}  

$C(x,\xi)=C(y,\eta)$ if and only of there is a lightlike geodesic joining  $x$ and $y$, that is, $\gamma(0)=x$, $\gamma(1)=y$, so that 

(a) $x$ and $y$ are conjugate to each other on $\gamma$,

(b) $\xi_*=J'(0)+\lambda \dot\gamma(0)$, $\eta_*=J'(1)+\lambda \dot\gamma(1)$ with some $\lambda\in \R$, where $J$ is a Jacobi field with  $J(0)=J(1)=0$. 
\end{proposition}

The proposition follows from Lemma~\ref{lem_conj}. Note that the proposition is consistent with the observation that at every point of $\gamma$, its conormal bundle is two-dimensional: the Jacobi field $J$ in the lemma is scaled so that the proposition holds, and $\lambda$ is responsible for the second dimension.

Assume that $\gamma_0 : [0,\ell] \to M$ is a lightlike geodesic with endpoints outside $\Omega$, where $f\in \mathcal{E}'(\Omega)$. Assume that $x_1:= \gamma_0(s_1)$ and $x_2:=\gamma_0(s_2)$ are conjugate along $\gamma_0$. 
Let $V_j$, $j=1,2$ be conic sets in $T^*M$ defined as the covectors $(\gamma(s),\xi)$ for $s$ close to $s_j$ and $\xi$ any spacelike covectors at $\gamma(s)$. Then $C$ is a diffeomorphism from $V_j$ to its image if $V_j$ is small enough. We can choose $V_j$ so that $C(V_1)=C(V_2)=:\mathcal{V}\subset T^*\mathcal M$ and so that $V_j$ projected to the base $M$ is a neighborhood of $x_j$. Set $C_j=C|_{V_j}$, $j=1,2$ and define $C_{21}:= C_2^{-1}\circ C_1: V_1\to V_2$. Then $C_{21}$ is the canonical relation of $L_2^{-1}L_1$ where $L_j$ is  $L$ microlocalized to $V_j$, $j=1,2$.

Assume that $f=f_1+f_2$ with $f_j$ supported near $x_j$ but away from the endpoints of $\gamma_j$, $j=1,2$. 

\begin{theorem}\label{thm cancellation}
Suppose that $\kappa$ does not vanish near $x_1$ and $x_2$.
Let $f_j\in \mathcal{E}'(\Omega)$ with $\WF(f_j)\subset V_j$ with $V_j$ as above and small enough, $j=1,2$. Then 
\[
L(f_1+f_2)\in H^s(\mathcal{V})
\]
if and only if 
\[
f_2+L_2^{-1}L_1f_1\in H^{s-1/2}(V_1),
\]
where the inverses are microlocal parametrices. 
\end{theorem}

The proof is immediate given the properties of $L_j$ above which make $L_j$ elliptic FIOs of order $-1/2$ with diffeomorphic canonical relations.  The significance of the theorem is that given $f_1$ with spacelike singularities near $x_2$ in a neighborhood of the conormal bundle to $\gamma$ at $x_1$,  one can also construct $f_2$ singular near $x_2$ so that $L(f_1+f_2)$ is smooth. This statement is symmetric w.r.t.\ $s_1$ and $s_2$, of course. Therefore, the singularity  in the light ray transform  that is produced by $f_1$ is cancelled by the singularity produced by $f_2$.
On a manifold that contains many conjugate points, Theorem \ref{thm cancellation} can be considered as a cloaking result for the singularities. 
For instance, on a Lorentzian manifold $(M,g)$, that is conformal to the product space $(\R\times S^n,-dt^2+g_{{}_{S^n}})$,  any space-like element $((t_1,y_1),(\tau_1,\eta_1))\in \hbox{WF}(f_1)$ can be cancelled by a function $f_2$ that is supported near a point $(t_2,y_2)$,
where $t_2=t_1+ (2\pi+1) m,$ $m\in \mathbb Z$, and $y_2$ is an antipodal point to $y_1$. 
Also, observe that the function $f_2$ that hides an element of the wave front set of $f_1$ can be supported either in the future or in the past of the support of function  $f_1$. This has similar spirit to results on
cloaking for the Helmholtz equation by anomalous localized resonance \cite{MN}
and the active cloaking results  \cite{Guevara}, where scattered field produced by an object is cancelled by a metamaterial object or an active source that located near the object.

Theorem \ref{thm cancellation} also describes the microlocal kernel of $L$ in $V_1\cup V_2$.

\section{Applications} \label{sec_applications}
We discuss two application we already mentioned in the introduction. 

\subsection{A time dependent potential} Let $\Box_g$ be the wave operator related to a Lorentzian metric $g$ on $M$ with a timelike boundary $\bo$: 
\be{A1}
\Box_g u= |\det g|^{-\frac12}\partial_{x^i}({\ntext |\det g|^{\frac12}}g^{ij} \partial_{x^j}u).
\ee
One can introduce a magnetic field $A$ as in \cite{St-Yang-DN} but to keep it simple, we assume that $A=0$. 
We assume that there is a smooth real valued function $t$ so that its level sets are compact and spacelike. Set $M_{ab}=\{a\le t\le b\}$ with some $a<b$. By \cite{Hormander3}, see also  \cite{St-Yang-DN}, the following problem is well posed
\be{A2}
(\Box_g+q)u=0 \ \ \text{in $M$}, \quad u|_{t<a}=0, \quad u|_{\bo}=f,
\ee
where $q$ is a smooth potential in $M$ and $f\in H^s(\bo)$, $s\ge1$ is a given function with $f=0$ for $t<0$. We have $u\in H^s(M)$ and the DN map $ \Lambda:  f\mapsto \partial_\nu u|_{\bo}$ is well defined, where $\partial_\nu$ is the conormal derivative. As shown in \cite{St-Yang-DN}, the second term in the singular expansion of $\Lambda$ recovers algorithmically and stably the light ray transform $Lq$ of $q$. By our results, in absence of conjugate points along lightlike geodesics, one can recover (stably) the spacelike singularities of $q$. We can interpret those as the singularities moving slower that light.

A special case is to assume that $M$ is the cylinder $M=[a,b]\times M'$ where $M'$ is compact with a boundary and $g=-\d t^2+g'(x)$, where $g_0$ is a (time independent) Riemannian metric on $M'$. Assume that $q=q(t,x)$, $x\in M$. Then the future pointing lighlike geodesics in $M$ are given by $(t,\gamma'(t))$, up to a reparameteriation, where $\gamma'$ are unit speed geodesics in $M$. Then $\Lambda$ recovers 
\[
Lq = \int q(t,\gamma'(t))\,\d t
\]
for various $\gamma'$. This transform has been studied in 
\cite{Feizmohammadi2019}.
An even more special case is to assume that $g'$ is Euclidean. This leads us to the Minkowski light ray transform studied in  section~\ref{sec_M}. This problem  was considered in \cite{MR1004174}.

\subsection{Time dependent speed} Assume again that $M$ is the cylinder $M=[a,b]\times M'$ where $M'$ is compact with a boundary and $g=-\d t^2+g'(t,x)$, with  $g'$ a Riemannian metric on $M'$ depending smoothly on the time variable. Here and below, we follow the same notational convention as above --- primes denote projections onto the last $n$ components of the $1+n$ dimensional vectors of covectors. 
Locally, every Lorentzian metric can be put in this form,.
As shown in \cite{St-Yang-DN}, the  Dirichlet-to-Neumann map $\Lambda$ for the wave operator $\square_g+q$ is an FIO of order zero away from the diagonal,
with the canonical relation equal to the lens relation $\mathcal{L}$ associated with $g$. In particular, we recover $\mathcal{L}$. The linearization of $\mathcal{L}$ near a fixed $g'$ is a light ray transform but it involves derivatives of the perturbation  $\delta g'$, see, for example, \cite{SUV_localrigidity} for the time independent case. Instead of linearizing $\mathcal{L}$, we will linearize the travel times between boundary points, defined locally as we explain below. 

Let $(t_1,x_1)$ and $(t_2,x_2)$ be the endpoints of a lightlike geodesic $\gamma_0(t)$ in $M$, transversal to the boundary at both ends, with $t_1<t_2$ and $x_1\in\bo_1$, $x_2\in\bo_1$. 
We parameterize the lightlike geodesics $\gamma(s)$ near $\gamma_0$ by initial points $(t_1,x)$ on $t=t_1$ (here, $x$ plays the role of $z$ before), and we require $\dot \gamma(0)=(1,\theta)$ where $\gamma'(0)=\theta$ must be unit in the metric $g'(t_1,\cdot)$. 
Assume now  that  $(t_1,x_1)$ and $(t_2,x_2)$ are not conjugate along $\gamma_0$. Fix $(t_1,x_1)$ and denote temporarily the geodesics issued from this point in the direction $(1,\theta)$ by $\gamma(s,\theta)$.

As before, we can parameterize $\theta$ locally by $a\in \R^{n-1}$. By the non-conjugate assumption, the differential $\d\gamma(s,\theta(a))/\d (s,a)$ is injective at $(s,a) =(0,a_0)$, where $a_0$ corresponds to $\theta_0:=\dot\gamma_0'(0)$. By Lemma~\ref{lemma_J} its range is $\dot\gamma(t_2)^\perp$,  which is also the tangent space  at $\gamma(s,\theta(a))$ of the lighlike cone (or flowout) with vertex $\gamma(0,\theta(a_0))$. The  projection of  $\dot\gamma(t_2)^\perp$ to its last  $n$ variables, i.e., to the tangent space spanned by $ \partial_{x^1},\dots,\partial_{x^n}$, is in direction transversal to $\dot\gamma(t_2)^\perp$,  since $\partial/\partial t$   does not belong to the latter; hence it is also a hyperplane of dimension $n$.   
Therefore,  $\d\gamma'(s,\theta(a)))/\d (s,a)$ is an invertible $n\times n$ Jacobian, and the map $(s,\theta)\mapsto \gamma_0'(s,\theta)$ is a local diffeomorphism, smoothly depending on $(t_1,x_1)$. In particular, given $x$ close to $x_2$, one can define the local travel time $\tau(t_1,x_1,x)$ from $(t_1,x_1)$ to $x$ by setting $x=\gamma_0'(s,\theta)$, solving for $s$ and $\theta$ and then plugging them into the zeroth component of $\gamma_0(s,\theta)$. Restricting $x$ to $\bo'$, we get the travel times $\tau(t,x,y)$ (since we can vary $(t_1,x_1)$ as well) for $(t,x,y)$ close to $(t_1,x_1,x_2)$, satisfying  $\tau(t_1,x_1,x_2)=t_2$. Note that we defined $\tau$ using geodesics close to $\gamma_0$ only. In the applied literature, those times are also called arrival times since they correspond to times a wave produced by a point source at $(t,x)$ arrives at $y$. 

Assume now that we have a fixed background $g'_0$ which is stationary, i.e.,   $g'_0=g'_0(x)$ and $g_0:=-\d t^2+g_0'(x)$. Then the future pointing light geodesics for $g_0$, parameterized as above, take the form $\gamma(s)=(s,\gamma'(s))$, where $\gamma'$ are unit speed geodesics in the metric $g'_0$. The non-conjugacy assumption we made is equivalent to $x_1$ and $x_2$ not being conjugate along $\gamma_0'$. Then $\tau(t,x_1,y)=\tau'(x_1,y)+t$, where $\tau'$ is the localized (Riemannian) travel time defined similarly to the one above, see also \cite{SU-lens}.

We want to linearize the travel times $\tau$ for a family of metrics $g_\epsilon = -\d t^2+g_\epsilon'(t,x)$ near $g_0$. 
We write $\gamma(s,\epsilon)$ for the null geodesics associated to $g_\epsilon$
and use an analogous notation for the local travel times. 
Let $\gamma(s,\eps)$ be a smooth  variation of $\gamma_0(s)$ with $|\eps|\ll1$ so that $\gamma(s,0)=\gamma_0(s)$; with the same endpoints for all $s$ in the following sense: 
\be{A2'}
\gamma(0,\eps)=(t_1,x_1)\quad \text{and}\quad \gamma(t_2-t_1,\eps)= (\tau(t_1,x_1,x_2,\eps),x_2)
\ee
with $\tau(t_1,x_1,x_2,0)= \tau_0(t_1,x_1,x_2)$, where $\tau_0$ corresponds to $g_0$. 
The second identity in \r{A2'} says {\ntext that $\gamma'(s)=x_2$} for $s=t_2-t_1$ (rather for some $\epsilon$--dependent $s$). This can always be achieved by parameterizing $\gamma(\cdot,\eps)$ appropriately, depending on $\eps$.

Set 
\begin{eqnarray}\label{A3}
& &E_0(\eps):= \int_{0}^{t_2-t_1}(g_0)_{ij}(\gamma(s,\eps)) \dot \gamma^i(s,\eps)\dot \gamma^j(s,\eps)\, \d s,\\ 
& &E(\eps):= \int_{0}^{t_2-t_1}(g_\eps)_{ij}(\gamma(s,\eps)) \dot \gamma^i(s,\eps)\dot \gamma^j(s,\eps)\, \d s, \nonumber
\end{eqnarray}
where the integrands are written in local coordinates. Then writing $D_\eps \d_t= D_t\d_\eps$, with $D_\eps$ and $D_s$ being covariant derivatives with respect to metric $g_0$, and integrating by parts, we get 
\be{A4}
\begin{split}
 E_0'(0) &=  2\int_{0}^{t_2-t_1}g_0(D_\eps|_{\eps=0} \dot \gamma(s,\eps), \dot \gamma(s,0)) \, \d s\\
&=  2\int_{0}^{t_2-t_1}g_0(D_s \d_\eps|_{\eps=0} \gamma(s,\eps), \dot \gamma_0(s)) \, \d s\\
&=  2\int_{0}^{t_2-t_1}\d_s g_0(\d_\eps|_{\eps=0} \gamma(s,\eps), \dot \gamma_0(s)) \, \d s\\
& =-2 \d_\eps|_{\eps=0}\tau(t_1,x_1,x_2,\eps).
\end{split}
\ee
The last equality follows by differentiating the second identity in \r{A2'} w.r.t.\ $\eps$ at $\eps=0$.

 As curves $\gamma(s,\eps)$  are null-geodesics with respect the metric
  $g_\epsilon$, we see that $E(\eps)=0$. Writing {\ntext $g_\epsilon'(t,x) = g_0'(x) + \epsilon h(t,x)$ and} using the calculation in \eqref{A4}, we get
  \be{A4 mod}
\begin{split}
 0&=E'(0) =  2\int_{0}^{t_2-t_1}g_0(D_\eps|_{\eps=0} \dot \gamma(s,\eps), \dot \gamma(s,0)) \, \d s + 
 \d_\eps\int_0^{t_2-t_1}g_\epsilon(\dot \gamma_0(s),\dot\gamma_0(s))\, \d s\bigg|_{\eps=0} \\
&=  -2 \d_\eps|_{\eps=0}\tau(t_1,x_1,x_2,\eps)+ \int_0^{t_2-t_1}h(\dot \gamma_0(s),\dot\gamma_0(s))\, \d s.
\end{split}
\ee
This  yields
  \be{A4 mod2}
\frac{\d}{\d \eps}\Big|_{\eps=0}\tau(t_1,x_1,x_2,\eps) = \frac12 \int_0^{t_2-t_1}h(\dot \gamma_0(s),\dot\gamma_0(s))\, \d s.
\ee
Therefore, the linearization of the travel times, up to the constant factor $1/2$ is the tensorial lightlike transform written in local coordinates in the form
\[
L^{(2)} h(\gamma) = \int h_{ij}(\gamma(s))\dot\gamma^i(s) \dot\gamma^j(s)\,\d s,
\]
where, in this particular application, the symmetric tensor $h$ satisfies $h_{0j}=0$. Recall that $\gamma$ runs over null geodesics for the metric $g_0=-\d t^2+g_0'(x)$ between points of $[a,b]\times \partial M'$.  In particular, if $g_0'(t,x)=c_0^{-2}(x)\d x^2$, and {\ntext 
\[g_\epsilon=-dt^2+\frac 1{(c_0(x)+(\delta c)(t,x))^2}\d x^2,\]
is a the Lorentzian metric corresponding to the perturbed time-dependent speed $c_0(x)+(\delta c)(t,x)$, then in linearization, we get the scalar light ray transform $Lf$, see \r{L_local},  of $$f(t,x):=-2c_0^{-3}(x)\delta c(t,x).$$ The problem of recovering the perturbation
of the wave speed $\delta c$  is encountered in the ultrasound imaging methods in medical imaging.
When $\delta c=\delta c(x)$ is independent of time, the waves that travel through the medium and collect information along geodesics $\gamma$ of $g_0'(x)$ are used in Transmission Ultrasound Tomography. This imaging modality has been used since the pioneering study of J. Greenleaf \cite{JGreenleaf} on 1980's. 
The case when the perturbation of the wave speed  $\delta c(x,t)$ depends on  time
is studied in Doppler ultrasound tomography, see \cite{Oglat,Jensenbook} and references there in.
The methods developed in this paper could be applicable in transmission ultrasound imaging of moving
tissues and organs, e.g.\ in the analogous imaging tasks where the backscattering measurements are presently used in Doppler echocardiography, where  the Doppler ultrasound tomography
is used to examine the heart \cite{Ommen}.}



\end{document}